\DeclareSymbolFont{AMSb}{U}{msb}{m}{n}
\documentclass[noamsfonts]{amsart}

\usepackage[usenames,dvipsnames]{xcolor}
\definecolor{cite}{HTML}{0851A6}
\definecolor{url}{HTML}{0851A6}
\definecolor{link}{HTML}{8F0C00}

\usepackage[alphabetic]{amsrefs}

\usepackage[colorlinks=true, linkcolor=link, citecolor = cite, linktocpage]{hyperref}

\usepackage{fancyhdr, amsmath, amsthm, enumerate, microtype}
\usepackage[all, cmtip]{xy}

\usepackage[utf8]{inputenc}
\usepackage[bitstream-charter]{mathdesign}
\usepackage[T1]{fontenc}

\DeclareMathAlphabet{\eur}{U}{zeus}{m}{n}
\newcommand{\matheur}[1]{\eur{#1}}

\theoremstyle{plain}
\newtheorem{prop}[subsubsection]{Proposition}
\newtheorem{lem}[subsubsection]{Lemma}
\newtheorem{cor}[subsubsection]{Corollary}
\newtheorem{thm}[subsubsection]{Theorem}
\newtheorem*{thm*}{Theorem}

\theoremstyle{definition}
\newtheorem{defn}[subsubsection]{Definition}

\theoremstyle{remark}
\newtheorem{rmk}[subsubsection]{Remark}
\newtheorem{expl}[subsubsection]{Example}

\newcommand{\teq}{\addtocounter{subsubsection}{1}\tag{\thesubsubsection}}

\newcommand{\all}{\mathrm{all}}
\newcommand{\Cat}{\matheur{C}\mathrm{at}}
\DeclareMathOperator{\Conf}{Conf}
\DeclareMathOperator*{\colim}{colim}
\newcommand{\ch}{\mathrm{ch}}
\DeclareMathOperator{\ComCoAlg}{ComCoAlg}
\newcommand{\ComCoAlgch}{\ComCoAlg^\ch}
\newcommand{\ComCoAlgstar}{\ComCoAlg^\star}
\newcommand{\cont}{\mathrm{cont}}
\DeclareMathOperator{\Corr}{Corr}
\DeclareMathOperator{\CorrPreStkSchAll}{\Corr(\PreStk)_{\sch;\all}}

\newcommand{\DGCat}{\mathrm{DGCat}}
\newcommand{\DGCatpres}{\DGCat_\pres}
\newcommand{\DGCatprescont}{\DGCat_{\pres, \cont}}
\DeclareMathOperator{\diag}{diag}
\newcommand{\disj}{\mathrm{disj}}
\newcommand{\Disk}{\mathrm{Disk}}
\DeclareMathOperator{\ev}{ev}
\newcommand{\et}{\mathrm{\acute{e}t}}
\newcommand{\etale}{\'etale}
\DeclareMathOperator{\Fact}{Fact}
\newcommand{\Fq}{\mathbb{F}_q}
\newcommand{\Fqbar}{\lbar{\mathbb{F}}_q}
\DeclareMathOperator{\Free}{Free}
\newcommand{\fSet}{\mathrm{fSet}}
\DeclareMathOperator{\Fun}{Fun}
\DeclareMathOperator{\Gal}{Gal}
\newcommand{\gr}{\mathrm{gr}}
\newcommand{\id}{\mathrm{id}}
\newcommand{\iso}{\mathrm{iso}}
\DeclareMathOperator{\ins}{ins}
\newcommand{\Kunneth}{K\"unneth}
\DeclareMathOperator{\Lie}{Lie}
\DeclareMathOperator{\Liech}{\Lie^\ch}
\DeclareMathOperator{\Liestar}{\Lie^\star}
\DeclareMathOperator{\LKE}{LKE}

\newcommand{\oblv}{\mathrm{oblv}}
\newcommand{\op}{\mathrm{op}}
\newcommand{\otimesch}{\otimes^\ch}
\newcommand{\otimesshriek}{\overset{!}{\otimes}}
\newcommand{\otimesstar}{\otimes^\star}
\DeclareMathOperator{\PConf}{PConf}
\newcommand{\pres}{\mathrm{pres}}
\newcommand{\PreStk}{\mathrm{PreStk}}
\newcommand{\Prim}{\mathrm{Prim}}
\newcommand{\proper}{\mathrm{proper}}
\newcommand{\qandq}{\qquad\text{and}\qquad}
\newcommand{\Ql}{\mathbb{Q}_\ell}
\newcommand{\Qlbar}{\lbar{\mathbb{Q}}_\ell}
\DeclareMathOperator{\Ran}{Ran}
\newcommand{\res}{\mathrm{res}}
\DeclareMathOperator{\Sym}{Sym}
\newcommand{\sch}{\mathrm{sch}}
\newcommand{\Sch}{\mathrm{Sch}}
\DeclareMathOperator{\Spec}{Spec}
\newcommand{\surjects}{\twoheadrightarrow}
\DeclareMathOperator{\Shv}{Shv}
\newcommand{\sm}{\mathrm{sm}}
\newcommand{\Spc}{\mathrm{Spc}}
\newcommand{\sqto}{\leadsto}
\DeclareMathOperator{\Supp}{Supp}
\newcommand{\surj}{\mathrm{surj}}
\newcommand{\union}{\mathrm{union}}
\newcommand{\Vect}{\mathrm{Vect}}

\newcommand{\lbar}[1]{\overline{#1}}
\newcommand{\oversetsupscript}[3]{\overset{#2}{#1}{}^{#3}}
\newcommand{\arrdis}{0.4ex} 

\title[Free factorization algebras in algebraic geometry]{Free factorization algebras and homology of configuration spaces \\ in algebraic geometry}
\author{Q.P. H\`{\^o}}
\date{\today}


\begin{document}
\maketitle

\begin{abstract}
	We provide a construction of free factorization algebras in algebraic geometry and link factorization homology of a scheme with coefficients in a free factorization algebra to the homology of its (unordered) configuration spaces. As an application, this construction allows for a purely algebro-geometric proof of homological stability of configuration spaces.
\end{abstract}

\tableofcontents

\section{Introduction}
The theory of factorization algebras has its root in vertex algebras and was first formulated in the beautiful language of algebraic geometry in the case of curves by Beilinson and Drinfel'd in~\cite{beilinson_chiral_2004}. Further developments to generalize these results to higher dimensional schemes required many ideas and techniques of a homotopical nature, and were carried out by Francis and Gaitsgory in~\cite{francis_chiral_2011} several years later. These advances have many applications in the Geometric Langlands program, and very recently, have culminated in the proof of Weil's conjecture on the Tamagawa number~\cites{gaitsgory_atiyah-bott_2015,gaitsgory_weils_2014}.

A topological version of the theory, known as topological factorization algebras/homology and $E_n$-algebras, was developed by Lurie in~\cite{lurie_higher_2016}. This was further developed by Ayala and Francis in~\cites{ayala_factorization_2012,ayala_poincare/koszul_2014}. The theory of topological factorization homology provides an efficient tool for attacking classical questions related to the stability of homology groups of a family of spaces (see, for example,~\cites{knudsen_betti_2014,kupers_homological_2013,kupers_$e_n$-cell_2014}). In some sense, this is not surprising, since one can see various hints from previous work. One example is the appearance of a Lie algebra structure when one studies the homology groups of configuration spaces (see, for example,~\cites{totaro_configuration_1996, getzler_resolving_1999}). In the setting of factorization algebras, this could be conceptually understood as an instance of the chiral Koszul duality theory developed in~\cite{francis_chiral_2011}, one of the main tools employed in the current paper.

Starting from~\cite{beilinson_chiral_2004}, most of the existing literature, with the exception of~\cites{gaitsgory_atiyah-bott_2015,gaitsgory_weils_2014}, works exclusively with $D$-modules. In this setting, however, one cannot construct free factorization algebras, due to, as we shall see, the lack of the proper-pushforward functor in general. Free factorization algebras are nonetheless one of the most basic objects in the topological setting, and they are used extensively, for example, in~\cites{knudsen_betti_2014,kupers_$e_n$-cell_2014}. 

Working in the context of constructible sheaves on schemes with values in vector spaces over a field of characteristic 0 (we refer the reader to \S\ref{subsec:sheaves_on_schemes} for what we mean by sheaves), this paper provides a construction of free factorization algebras (or $E_X$-algebras, where $X$ is a scheme). As a result, we provide a mechanism to translate all the techniques and results found in~\cite{knudsen_betti_2014} to a purely algebraic setting. In fact, the main result~\cite[Thm. 1.1]{knudsen_betti_2014} could be viewed as the special case of constant sheaves in our setting. Moreover, since our method is internal to the world of algebraic geometry, the main results are automatically compatible with Galois action from the base field.

One might wonder if it is possible to translate methods in~\cites{kupers_homological_2013,kupers_$e_n$-cell_2014} into algebraic geometry. We intend to return to this question in a future work.

\subsection{Factorization algebras} Let us offer an intuitive picture of the main object of this paper: factorization algebras. The reader should note that we only intend to give an impressionistic outline without the technical baggage. 

\subsubsection{Topological formulation} The topological avatar of factorization algebras is $E_n$-algebras, which has a very nice geometric interpretation. When $n=1$ and $n=\infty$, up to homotopy coherence, these are associative and commutative algebras respectively. How does multiplication work in these classical objects? Suppose we have two elements $a, b\in \matheur{A}$ where $\matheur{A}$ is a classical algebra. Then essentially, there are two ways to multiply them: $ab$ and $ba$---either $a$ is on the left of $b$ or vice versa. In other words, there are essentially two relative positions between $a$ and $b$, and for each relative position, we have a way to multiply.

Now, let $\matheur{A}$ be an $E_n$-algebra, and let $a, b \in \matheur{A}$. To multiply $a$ and $b$, we have to put them on $\mathbb{R}^n$ and for each relative position, we have a way of multiplying. So, the multiplication map is not of the form
\[
	\matheur{A}^2 \to \matheur{A}
\]
anymore. Instead, we need to keep track of the relative position and get a map of the form
\[
	\PConf_2 \mathbb{R}^n \times \matheur{A}^2 \to \matheur{A}.
\]
And suppose we need to multiply $k$ elements, we need to have the following map
\[
	\PConf_k \mathbb{R}^n \times \matheur{A}^k \to \matheur{A}.
\]
All these maps have to satisfy several coherence properties that we will not spell out here.

Here, 
\[
	\PConf_k \mathbb{R}^n = \oversetsupscript{(\mathbb{R}^n)}{\circ}{k}
\]
denotes the \emph{ordered} configuration space of $k$ distinct points on $\mathbb{R}^n$, where for any space $X$, $\oversetsupscript{X}{\circ}{k}$ denotes the open complement of the fat diagonal of $X^k$.

\subsubsection{Free $E_n$-algebras} In the case of associative algebras, suppose we start with an object $V$. Then to create a free associative algebra object, we formally add all the possible ways to multiply and get something of the form $\bigsqcup_n V^n$. The idea is similar for $E_n$-algebras. To create a free $E_n$-algebra out of $V$, we formally add all the possible ways to multiply, which includes all relative positions, and get something of the form
\[
	\bigsqcup_{k\geq 1} (\PConf_k \mathbb{R}^n \times V^k)/\Sigma_k,
\]
where the symmetric group on $k$ letters, $\Sigma_k$, acts diagonally.

Suppose we are working in the category of chain complexes over a field $F$. Then one can show that the free $E_n$-algebra generated by a chain complex $V$ is given by (see~\cite{ayala_factorization_2012})
\[
	\bigoplus_{k\geq 1} (C_*(\PConf_k \mathbb{R}^n, F) \otimes_F V^{\otimes k})_{\Sigma_k}.
\]

Such calculations could be globalized via the theory of factorization homology. Using this, one can prove, for instance, that
\[
	\int_M \Free_{E_n}(F) \simeq \bigoplus_{k\geq 1} C_*(\Conf_k M, F),
\]
where $M$ is an $n$ dimensional manifold, and the left hand side denotes the factorization homology of $M$ with coefficients in the free $E_n$-algebra generated by $F$. This is the starting point of~\cite{knudsen_betti_2014}.

\subsubsection{Sheaf theoretic formulation} The multiplication operation within an $E_n$-algebra could be captured using the language of sheaves on a space/scheme $X$. However, the situation is inherently global in this setting: instead of having $k$ points on $\mathbb{R}^n$, these points will be on $X$. We will thus call them $E_X$-algebras.

Let $\matheur{F} \in \Shv(X)$ be a sheaf on $X$. Then the operation of multiplying 2 elements (\`a la $E_n$-algebras) could be viewed as a map of sheaves
\[
	m: i^! j_! j^! \matheur{F}^{\boxtimes 2} \to \matheur{F}, \teq\label{eq:mult_map}
\]
where $i$ is the diagonal map and $j$ is the open complement as follows
\[
\xymatrix{
	X^2 \backslash X \ar[r]^{j} & X^2 & \ar[l]_{i} X.
}
\]
For instance, when $\matheur{F}$ is a constant sheaf on $\mathbb{R}^n$ then the homology of the configuration space of 2 points on $\mathbb{R}^n$ appears as the costalk at a point of the left hand side of~\eqref{eq:mult_map}. 

One can use the multiplication map~\eqref{eq:mult_map} to construct a sheaf $\matheur{F}^{(2)} \in \Shv(X^2)$ such that
\[
	i^! \matheur{F}^{(2)} \simeq \matheur{F} \qquad\text{and}\qquad j^! \matheur{F}^{(2)} \simeq j^! \matheur{F}^{\boxtimes 2}.
\]
As in the case of $E_n$-algebras, one also has to consider $X^n$ for all $n\geq 1$, and a similar argument would give a sequence of sheaves $\matheur{F}^{(n)} \in \Shv(X^n)$ which satisfy similar compatibility conditions as above. The construction of free $E_X$-algebras would then involve formally taking direct sums of sheaves appearing on the left hand side of~\eqref{eq:mult_map} (for all $n\geq 1$), and then gluing them together in a proper way. 

We will not directly follow this route, as we would very quickly run into a huge combinatorial mess. Instead, we will reformulate the construction of free $E_n$-algebras in a more categorical manner (see \S\ref{subsec:quick_review_E_n}), which allows us to adapt to the sheaf theoretic setting. We will then use the theory originally developed in~\cite{beilinson_chiral_2004} and~\cite{francis_chiral_2011} to gracefully handle the combinatorial complexity that arises.

\subsection{Summary of main results} We will now give a summary of the main results. The precise statements of these results will be given in the body of the paper.

As mentioned above, the main goal of this paper is to provide a construction of free factorization algebras in the context of constructible sheaves in algebraic geometry. Let $X$ be a scheme. Then we denote by $\Fact(X)$ the category of factorization algebras over $X$ (see \S\ref{sec:E_X_algebras} for the definition). As in the topological setting, there's a natural forgetful functor
\[
	\delta^!: \Fact(X) \to \Shv(X).
\]

Let
\[
	T: \Shv(X) \to \Shv(\Conf X)
\]
denote the functor that acts on sheaves by tensoring up (see \S\ref{subsubsec:1st_stage_T_Free_E_X}), and let
\[
	g_!: \Shv(\Conf X) \to \Shv(\Ran X)
\]
be the pushforward along the natural map (see \S\ref{subsubsec:various_spaces})
\[
	g: \Conf X \to \Ran X.
\]

The main result of this paper is the following
\begin{thm}[Theorem~\ref{thm:Free_E_X} \& Proposition~\ref{prop:2_ways_E_X}] \label{thm:intro:Free_E_X}
	We have a pair of adjoint functors
	\[
		\Free_{E_X}: \Shv(X) \rightleftarrows \Fact(X): \delta^!.
	\]
	where $\Free_{E_X} = g_!\circ T$.
\end{thm}

Note that part of the content of this theorem is the fact that $g_!\circ T$ actually factors through $\Fact(X)$, since a priori, the target of this functor is just $\Shv(\Ran X)$. In fact, a large part of this paper is used to establish technical results needed for this verification.

We will also provide a second construction of free factorization algebras via free Lie algebras (see Proposition~\ref{prop:2_ways_E_X}), where we make use of~\cite{francis_chiral_2011} to do the heavy lifting. These two constructions come with natural gradings: the first one comes from the cardinality of configuration, and the second from powers of the Lie generators, which will be called the Lie grading. We will show that

\begin{thm}[Theorem~\ref{thm:compatibility_of_gradings}] \label{thm:intro:compatibility_of_gradings}
	These two gradings are the same.
\end{thm}

The first grading links to the cardinality of configuration spaces, which is what we are interested in. On the other hand, the second construction provides an organizing tool to assemble all the cohomology of configuration spaces of $X$ together.

As a consequence of the two results above, we get the following 
\begin{prop}[Proposition~\ref{prop:coh_of_conf_spaces}] \label{prop:intro:coh_of_conf_spaces}
	Let $X$ be a scheme, and $\matheur{F} \in \Shv(X)$. Then, there exists a functorial quasi-isomorphism of chain complexes
	\[
		C^*_c(\Conf X, T\matheur{F}) = \bigoplus_{n\geq 1} C^*_c(\Conf_n X, (T\matheur{F})_n) \simeq C_*^{\Lie}(C^*_c(X, \Free_{\Lie}(\matheur{F}[-1]))),
	\]
	which exchanges the cardinality grading (of configuration) on the left hand side with the Lie grading on the right hand side.	
\end{prop}

Results in~\cite{knudsen_betti_2014} could now be done internally in the world of algebraic geometry. For example, the proof found in~\cite[Sect. 5.3]{knudsen_betti_2014}, which is an argument about the homology of a Lie algebra (namely, the one on the right hand side of Proposition~\ref{prop:intro:coh_of_conf_spaces}), could now be copied without any modification to yield a proof in our context. As a consequence, we get
\begin{cor}[Corollary~\ref{cor:homological_stability}] \label{cor:intro:homological_stability}
	For a connected smooth scheme $X$ of dimension $n\geq 1$, cap product with the unit in $H^0(X, \Lambda)$ induces a map
	\[
		H^*_c(\Conf_{k+1} X, \omega_{\Conf_{k+1} X}) \to H^*_c(\Conf_k X, \omega_{\Conf_k X})
	\]
	that is 
	\begin{enumerate}[\quad --]
		\item an isomorphism, for $*>-k$, and a surjection for $*=-k$, when $X$ is an algebraic curve; and
		\item an isomorphism for $*\geq -k$, and a surjection for $*=-k-1$ in all other cases.
	\end{enumerate}
\end{cor}

Moreover, since our theory is developed within the world of algebraic geometry, Galois actions are already part of the output. For instance, the equivalence in Proposition~\ref{prop:intro:coh_of_conf_spaces} and the stabilizing maps in Corollary~\ref{cor:intro:homological_stability} are compatible with the Galois actions. For a more precise discussion, we refer the reader to \S\ref{sec:consequences}.

\begin{rmk}
	When $X$ is a smooth scheme over $\mathbb{Z}[1/N]$, smoothly compactifiable with normal crossing complement, Farb and Wolfson have recently proved in~\cite{farb_etale_2015} $S_n$-representation stability of the cohomology of ordered configuration spaces of $X \otimes K$ (where $K = \lbar{\mathbb{Q}}$ or $K = \Fqbar$) that is also compatible with the Galois actions. For such a scheme $X$, this result is a generalization of Corollary~\ref{cor:intro:homological_stability}. Their technique is different from ours: it makes use of the theory of FI-modules, and moreover, relies on the comparison theorem between \etale{} cohomology and singular cohomology, which is the source of the extra assumptions mentioned above.
\end{rmk}

\subsection{Contents} We will now give a quick overview of the paper section-by-section.

\S\ref{sec:sheaves_on_prestacks} collects the technical results about sheaves on prestacks we need for the actual construction. The enthusiastic reader can jump directly to the constructions in \S\ref{sec:E_X_algebras} and \S\ref{sec:alternative_construction_Lie}, and backtrack when necessary. We start with a review of the theory of sheaves on prestacks, as developed in~\cite{gaitsgory_atiyah-bott_2015}. We recall the notion of pseudo-properness of a morphism between prestacks and the accompanying base change result. We will then construct an adjoint pair $f^* \dashv f_*$, extending the usual pullback and pushforward of sheaves on schemes. After that, we will introduce the technically important notions of pseudo-Artin prestacks and pseudo-smooth morphisms between them. The accompanying base change result, which is the generalization of the smooth base change theorem, is then proved.

\S\ref{sec:E_X_algebras} and \S\ref{sec:alternative_construction_Lie} carry out the main constructions of free factorization algebras. The first construction is given in Theorem~\ref{thm:Free_E_X}, while the second one in Proposition~\ref{prop:2_ways_E_X}. The compatibility of the two gradings is established in Theorem~\ref{thm:compatibility_of_gradings}.

We start \S\ref{sec:E_X_algebras} with a review of the topological picture, which serves as the main motivation for the actual construction. Then, we provide a slight extension to the notion of the chiral monoidal structure found in~\cite{francis_chiral_2011} and use it as an organizing tool to prove various technical properties of the construction. 

\S\ref{sec:alternative_construction_Lie} provides another construction of free factorization algebras via Lie algebras, which is equivalent to the first one due to formal reasons. The outputs of each of these constructions carry natural gradings. We conclude the section by showing that these two gradings agree.

\S\ref{sec:consequences} lists various direct consequences of our construction.

\subsection{Conventions and notation}
\subsubsection{Category theory}
We use $\DGCat$ to denote the $(\infty, 1)$-category of stable infinity categories, $\DGCatpres$ to denote the full subcategory of $\DGCat$ consisting of presentable categories, and $\DGCatprescont$ to denote the (non-full) subcategory of $\DGCatpres$ where we restrict to continuous functors, i.e. those commuting with colimits. $\Spc$ will be used to denote the category of spaces, or $\infty$-groupoids. 

\subsubsection{Algebraic geometry}
Throughout this paper, $k$ will be an algebraically closed ground field. We will denote by $\Sch$ the $\infty$-category obtained from the ordinary category of separated schemes of finite type over $k$. All our schemes will be objects of $\Sch$. In most cases, we will use the calligraphic font to denote prestacks, for instance $\matheur{X}, \matheur{Y}$ etc., and the usual font to denote schemes, for instance $X, Y$ etc.

\subsection{Acknowledgment} The author would like to express his deepest gratitude to D. Gaitsgory for his tireless patience and enthusiasm in explaining many results in~\cite{gaitsgory_atiyah-bott_2015} and also many other parts of mathematics. 

The author would like to thank E. Elmanto for his constant enthusiasm, and countless stimulating conversations. 

The author is grateful to his advisor B.C. Ng\^o for many years of guidance, support and patience.

The author would like to thank B. Knudsen for many helpful comments on previous drafts of this paper. 

Finally, the author is grateful to the anonymous referee for providing many comments which greatly improve the clarity and precision of the exposition.

\section{Sheaves on prestacks} \label{sec:sheaves_on_prestacks}
The theory of sheaves on prestacks has been developed in~\cite{gaitsgory_atiyah-bott_2015} and~\cite{gaitsgory_weils_2014}. We will start this section with a brief review of this theory, which includes the definition of the category of sheaves on a prestack, as well as the existence and properties of the two functors $f_!$ and $f^!$. After that, we will define two new functors $f^*$ and $f_*$ for a special class of morphisms $f$ between prestacks. Finally, we prove various base change type results, and then conclude with a brief discussion on how some of the functors we have developed so far compose.

\subsection{Sheaves on schemes} \label{subsec:sheaves_on_schemes} We will adopt the same conventions as in~\cite{gaitsgory_atiyah-bott_2015}, except that we restrict ourselves to the ``constructible setting.'' 

\subsubsection{} In this paper, we employ a theory of constructible sheaves on schemes, with the 6 functor formalisms:
\begin{enumerate}[\quad (i)]
	\item When the ground field is $\mathbb{C}$, and $\Lambda$ is an arbitrary field of characteristic 0, we can take $\Shv(S)$ to be the ind-completion of the category of constructible sheaves on $S$ with $\Lambda$-coefficients.
	
	\item For any ground field $k$ in general, and $\Lambda = \Ql, \Qlbar$, we take $\Shv(S)$ to be the ind-completion of the category of constructible $\ell$-adic sheaves on $S$ with $\Lambda$-coefficients.
\end{enumerate}

See~\cite{liu_enhanced_2012} and~\cite{liu_enhanced_2014} for a fully developed theory.

\subsubsection{} We quickly recall the formal properties of these sheaf theories. Informally, for each scheme $S$, we assign to it the category of sheaves over it, denoted by $\Shv(S) \in \DGCatpres$, such that for each morphism of schemes
\[
	f: S_1\to S_2,
\]
we have the two usual pairs of adjunctions 
\[
	f^*: \Shv(S_2) \rightleftarrows \Shv(S_1): f_*
\]
and
\[
	f_!: \Shv(S_1) \rightleftarrows \Shv(S_2): f^!.
\]
Moreover, for
\[
	\matheur{F}_1, \matheur{F}_2 \in \Shv(S),
\]
we can form
\[
	\matheur{F}_1 \otimes \matheur{F}_2, \matheur{F}_1 \otimesshriek \matheur{F}_2 \in \Shv(S).
\]
Observe that these operations are endowed with a homotopy-coherent system of compatibilities for compositions of morphisms. Moreover, $f^!$ commutes with colimits.

\subsubsection{} More formally, our sheaf theory is given by functors
\[
	\Shv_*: \Sch \to \DGCatprescont \qandq \Shv^*: \Sch^\op \to \DGCatprescont
\]
as well as these following functors obtained by adjunctions
\[
	\Shv^!: \Sch^\op \to \DGCatprescont \qandq \Shv_!: \Sch \to \DGCatpres.
\]
For each scheme $S$,
\[
	\Shv^!(S) = \Shv^*(S) = \Shv_*(S) = \Shv_!(S) = \Shv(S).
\]
Thus, unless we want to emphasize which functor we use to move between different schemes, we will just write $\Shv(S)$.

Recall that $\DGCatprescont$ is equipped with a symmetric monoidal structure. Consider $\Sch$ equipped with the Cartesian symmetric monoidal structure. Then $\Shv^!$ and $\Shv^*$ are endowed with the right lax symmetric monoidal structure, which agree on values. Namely, for schemes $S_1, S_2$, we have a functor
\begin{align*}
	\Shv(S_1) \otimes \Shv(S_2) &\to \Shv(S_1\times S_2) \\
	(\matheur{F}_1, \matheur{F}_2) &\mapsto \matheur{F}_1\boxtimes \matheur{F}_2,
\end{align*}
equipped with a homotopy-coherent system of compatibilities for both $\Shv^!$ and $\Shv^*$.

In particular, for any scheme $S$, the functors $\diag^*$ and $\diag^!$, induced by the diagonal map $\diag: S\to S\times S$, equip $\Shv(S)$ with two separate symmetric monoidal structures
\[
	\matheur{F}_1, \matheur{F}_2 \mapsto \matheur{F}_1 \otimes \matheur{F}_2 := \diag^*(\matheur{F}_1\boxtimes \matheur{F}_2),
\]
and
\[
	\matheur{F}_1, \matheur{F}_2 \mapsto \matheur{F}_1 \otimesshriek \matheur{F}_2 := \diag^!(\matheur{F}_1\boxtimes \matheur{F}_2).
\]

The units of the $\otimes$- and $\otimesshriek$-monoidal structures on $\Shv(X)$ are $\Lambda_X$ (the constant sheaf) and $\omega_X$ (the dualizing sheaf) respectively.

\subsubsection{Base change and correspondences} \label{subsubsec:base_change_correspondences_schemes} The theory of sheaves developed in~\cites{liu_enhanced_2012, liu_enhanced_2014} in fact consists of a functor from the category of correspondences
\begin{align*}
	\Shv: \Corr(\Sch)_{\all;\all} &\to \DGCat_{\pres, \cont} \teq\label{eq:shv_sch_correspondences} \\
	S &\mapsto \Shv(S) 
\end{align*}
which sends a morphism in $\Corr(\Sch)_{\all;\all}$
\[
\xymatrix{
	M \ar[d]_v \ar[r]^h & X \\
	Y
} \teq\label{eq:corr_sch}
\]
to
\[
	v_* \circ h^!: \Shv(X) \to \Shv(Y).
\]
Here, the first (resp. second) subscript of $\Corr(\Sch)_{\all;\all}$ is used to denote that $v$ (resp. $h$) is allowed to be any morphism in $\Sch$. Moreover, the direction of the morphism between $X$ and $Y$ as pictured in~\eqref{eq:corr_sch} is that of
\[
	X \sqto Y.
\]
To avoid confusions with ``usual'' morphisms, we will use the squiggly arrow $\sqto$ to denote a morphism in the category of correspondences.

See~\cite[Intro. V]{gaitsgory_study_????} for a brief overview of correspondences and~\cite[Intro. V.1.3]{gaitsgory_study_????} for the translation  from the language used in~\cites{liu_enhanced_2012, liu_enhanced_2014} to that of correspondences.

\subsubsection{} Such a functor encodes the following base change isomorphism: suppose we have a pullback diagram of schemes
\[
\xymatrix{
	X' \ar[d]_f \ar[r]^g & X \ar[d]_f \\
	Y' \ar[r]^g & Y
}
\]
Then the fact that~\eqref{eq:shv_sch_correspondences} is a functor provides us with a natural equivalence
\[
	f_*g^! \simeq g^!f_*: \Shv(X) \to \Shv(Y').
\]

\subsubsection{} In what follows, we only need to use~\eqref{eq:shv_sch_correspondences} to construct a similar functor at the level of prestacks. This will then allow us to construct the \emph{chiral monoidal structures} on $\Ran$ and $\Conf$ in \S\ref{sec:E_X_algebras}.

\subsection{Prestacks} Recall that a prestack is a contravariant functor from $\Sch$ to $\Spc$. Namely, a prestack $\matheur{Y}$ is a functor
\[
	\matheur{Y}: \Sch^\op \to \Spc.
\]
We let $\PreStk$ denote the $\infty$-category of prestacks. Note that by Yoneda's lemma, we have a fully-faithful functor.
\[
	\Sch \hookrightarrow \PreStk.
\]

\subsection{Sheaves on prestacks} Sheaves on prestacks are defined in a formal yet straightforward way.

\subsubsection{} \label{subsubsec:def_sheaves_prestacks} For a prestack $\matheur{Y}$, the category $\Shv^!(\matheur{Y})$ is defined by
\[
	\Shv^!(\matheur{Y}) = \lim_{S\in (\Sch_{/\matheur{Y}})^\op} \Shv^!(S). \teq\label{eq:Shv_as_limits}
\]
Unwinding the definition, we see that an object $\matheur{F} \in \Shv(\matheur{Y})$ is the same as the following data
\begin{enumerate}[\quad (i)]
	\item A sheaf $\matheur{F}_{S, y} \in \Shv(S)$ for each $S\in \Sch$ and $y\in \matheur{Y}(S)$,
	\item An equivalence of sheaves $\matheur{F}_{S', f(y)} \to f^! \matheur{F}_{S, y}$ for each morphism of schemes $f: S'\to S$.
\end{enumerate}
Moreover, we require that this assignment satisfies a homotopy-coherent system of compatibilities.

\subsubsection{} More precisely, one can define $\Shv^!(\matheur{Y})$ as the right Kan extension of
\[
	\Shv^!: \Sch^\op \to \DGCatpres
\]
along the Yoneda embedding
\[
	\Sch^\op \to \PreStk^\op.
\]

\begin{rmk}
	For a prestack $\matheur{Y}$, the only notion of sheaf that we will use in this paper is that of $\Shv^!(\matheur{Y})$ as discussed above. Thus, we will write $\Shv(\matheur{Y})$ instead of $\Shv^!(\matheur{Y})$.
\end{rmk}

\subsection{Properties of prestacks and morphisms between them} \label{subsec:classes_of_prestacks_morphisms}
In the study of sheaves on schemes, special properties of schemes and morphisms between them usually translate to nice properties of the functors that act on the category of sheaves. Smooth and proper base change theorems are important examples of this pattern. In this subsection, we will introduce the names of the various special classes of prestacks as well as morphisms between them, which will be used throughout the paper.

\subsubsection{} Let $\gamma$ be a special class of morphisms between schemes that is preserved under compositions. Then we denote by $\Sch_\gamma$ the subcategory of $\Sch$ that has the same objects, but which is obtained from $\Sch$ by only allowing the 1-morphisms lying in $\gamma$.

\subsubsection{} For each $\gamma$ as above, we use $\PreStk'_\gamma$ to denote the category of functors
\[
	\Sch_\gamma^\op \to \Spc.
\]
Applying left Kan extension (see \S\ref{subsec:LKE}), we get a functor
\[
	\PreStk'_\gamma \to \PreStk,
\]
and let $\PreStk_\gamma$ denote the full subcategory of $\PreStk$ generated by the image of $\PreStk'_\gamma$.

\subsubsection{} Being a left adjoint, left Kan extension preserves colimits. Hence, a prestack $\matheur{Y}$ is in $\PreStk_\gamma$ if and only if $\matheur{Y}$ could be exhibited as a colimit of schemes
\[
	\matheur{Y} \simeq \colim_{\alpha} Z_\alpha
\]
where all the transition maps $Z_\alpha \to Z_\beta$ in the colimit are in $\Sch_\gamma$.

\subsubsection{} In this paper, the $\gamma$'s that we use are proper, smooth and \etale. A prestack $\matheur{Y}$ in $\PreStk_{\proper}, \PreStk_{\sm}$ or $\PreStk_{\et}$ is said to be a pseudo-scheme, pseudo-Artin prestack, or pseudo-DM prestack respectively.

\subsubsection{} If $\matheur{Y}_1$ is a prestack over $\matheur{Y}_2$, then we say that $\matheur{Y}_1$ is a pseudo-scheme, pseudo-Artin prestack or pseudo-DM prestack over $\matheur{Y}_2$ if the base change $\matheur{Y}_1 \times_{\matheur{Y}_2} S$ is an object in $\PreStk_\proper, \PreStk_\sm$ or $\PreStk_\et$ respectively. In the case of pseudo-scheme, we also say that the morphism $f: \matheur{Y}_1 \to \matheur{Y}_2$ is pseudo-schematic. Directly from the definition, these notions are preserved by arbitrary base change of prestacks.

\subsubsection{} A morphism $f: \matheur{Y}_1 \to \matheur{Y}_2$ between prestacks $\matheur{Y}_1, \matheur{Y}_2 \in \PreStk_\gamma$ is said to be pseudo-$\gamma$ if $f$ comes from $\PreStk_\gamma'$.

Unwinding the definition in the case where $\matheur{Y}_2 = S$ is a scheme, we see that $f$ being  pseudo-$\gamma$ is equivalent to it being possible to present $\matheur{Y}$ as a colimit of schemes
\[
	\matheur{Y} \simeq \colim_\alpha Z_\alpha \teq\label{eq:colim_pseudo_gamma}
\]
such that all the transition maps between the $Z_\alpha$'s, and the induced maps $Z_\alpha \to S$ are in $\Sch_\gamma$. 

\subsubsection{} In the case where $\gamma$ is proper, a morphism $f: \matheur{Y}_1 \to \matheur{Y}_2$ between arbitrary prestacks is said to be pseudo-proper if the pullback $\matheur{Y}_1 \times_{\matheur{Y}_2} S$ of $\matheur{Y}_1$ to any scheme $S$ over $\matheur{Y}_2$ is pseudo-proper over $S$.

\begin{rmk}
	It is tempting to define the notions of pseudo-smooth and pseudo-\etale{} for morphisms between arbitrary prestacks in an analogous manner as in the pseudo-proper case. Such a notion, however, is not well behaved. Namely, it might not be compatible with the special case defined for prestacks in $\PreStk_\gamma$. This is ultimately due to the fact that the left Kan extension $\PreStk'_\gamma \to \PreStk$ may fail to commute with finite limits in general. See also~\cite[Sect. 7.4]{gaitsgory_atiyah-bott_2015}.
\end{rmk}

\subsubsection{} A morphism $f: \matheur{Y}_1 \to \matheur{Y}_2$ between prestacks is said to be finitary pseudo-proper if the pullback $\matheur{Y}_1 \times_{\matheur{Y}_2} S$ to any scheme $S\in \Sch_{/\matheur{Y}_2}$ is a finite colimit
\[
	 \matheur{Y}_1 \times_{\matheur{Y}_2} S \simeq \colim_\alpha Z_\alpha
\]
where all the $Z_\alpha$'s are schemes proper over $S$.

\subsection{The adjoint pair $f^! \dashv f_!$} The construction of these functors are discussed in details in~\cite{gaitsgory_atiyah-bott_2015}. We will only summarize the results here.

\subsubsection{} From the construction of $\Shv(\matheur{Y})$, we get the functor
\[
	f^!: \Shv(\matheur{Y}_2) \to \Shv(\matheur{Y}_1)
\]
for free. Moreover, also from the construction, $f^!$ commutes with limits, and hence, it admits a left adjoint
\[
	f_!: \Shv(\matheur{Y}_1) \to \Shv(\matheur{Y}_2).
\]

In the case where
\[
	f: \matheur{Y} \to \Spec k,
\]
we write
\[
	C^*_c(\matheur{Y}, -) = f_!: \Shv(\matheur{Y}) \to \Shv(\Spec k) = \Vect_\Lambda.
\]

\subsubsection{} Since $\Shv$ for prestacks is constructed using right Kan extension, it turns a colimit of prestacks to a limit in $\DGCatpres$ (note that our functors are contravariant). Namely, if $\matheur{Y}$ is a colimit of prestacks, i.e. $\matheur{Y} \simeq \colim_\alpha \matheur{Y}_\alpha$, then
\[
	\Shv(\matheur{Y}) \simeq \lim_\alpha \Shv(\matheur{Y}_\alpha). \teq\label{eq:Shv_commutes_limits}
\]

\subsubsection{} By \S\ref{subsec:limit_vs_colimit}, we can alternatively represent $\Shv(\matheur{Y})$ as a colimit of categories in $\DGCatprescont$
\[
	\Shv(\matheur{Y}) \simeq \colim_\alpha \Shv(\matheur{Y}_\alpha) \teq\label{eq:Shv_as_colimits}
\]
using the fact that $f_!$ is the left adjoint to $f^!$. Let
\[
	i_\alpha: \matheur{Y}_\alpha \to \matheur{Y}
\]
denote the natural morphism. Then for any $\matheur{F} \in \Shv(\matheur{Y})$,
\[
	\matheur{F} \simeq \colim_\alpha i_{\alpha!} i^!_\alpha \matheur{F}.
\]

\begin{rmk} \label{rmk:colim_ev_ins_sheaves}
	For brevity's sake, we will write $\matheur{F}_\alpha$ or $i^!_\alpha \matheur{F}$, unless confusion might occur. This also conforms to the notation of \S\ref{subsubsec:def_sheaves_prestacks}. Thus, the equivalence above would become
	\[
		\matheur{F} \simeq \colim_\alpha i_{\alpha!} \matheur{F}_\alpha.
	\]
\end{rmk}

\subsubsection{} Note that for any prestack $\matheur{Y}$, we have an equivalence
\[
	\matheur{Y} \simeq \colim_{S\in \Sch_{/\matheur{Y}}} S.
\]
Using this presentation of $\matheur{Y}$, the equivalence~\eqref{eq:Shv_commutes_limits} is the same as~\eqref{eq:Shv_as_limits}.

\subsubsection{} In general, we don't have a good handle for $f_!$. However, when $S$ is a scheme, we have an explicit formula for it.

\begin{lem} \label{lem:compute_f_!_base_scheme}
	Let $f: \matheur{Y} \to S$ be a morphism from a prestack to a scheme, where $\matheur{Y} \simeq \colim_{\alpha} \matheur{Y}_\alpha$, and let $\matheur{F} \in \Shv(\matheur{Y})$. For each $\alpha$, we define various maps as in the following diagram
	\[
	\xymatrix{
		\matheur{Y}_\alpha \ar[r]^{i_\alpha} \ar[dr]_{f_\alpha} & \matheur{Y} \ar[d]^f \\
		& S
	}
	\]
	Then,
	\[
		f_! \matheur{F} \simeq \colim_{\alpha} f_{\alpha!} \matheur{F}_\alpha.
	\]
\end{lem}

\subsubsection{} Another situation where $f_!$ could be understood is when $f$ is pseudo-proper.

\begin{prop}[Pseudo-proper base change] \label{prop:pseudo_proper_base_change}
	Suppose we have the following pull-back square of prestacks
	\[
	\xymatrix{
		\matheur{X}' \ar[d]_f \ar[r]^g & \matheur{X} \ar[d]^f \\
		\matheur{Y}' \ar[r]^g & \matheur{Y}
	}
	\]
	where $f$ is pseudo-proper. Then the following natural transformation (obtained from the adjunction $f_! \dashv f^!$)
	\[
		f_! g^! \to g^! f_!
	\]
	is an equivalence.
\end{prop}

\subsection{The adjoint pair $f^* \dashv f_*$}  The functors $f^*$ and $f_*$ are not defined in general for an arbitrary morphism $f: \matheur{Y}_1 \to \matheur{Y}_2$ between prestacks. However, these functors are defined, and we have an adjoint pair $f^* \dashv f_*$ as usual, when $f$ is schematic.

\subsubsection{} Recall that a morphism of prestacks
\[
	f:\matheur{Y}_1\to \matheur{Y}_2
\]
is schematic if the base change of $\matheur{Y}_1$ to any scheme $S$ over $\matheur{Y}_2$ is a scheme.

\subsubsection{} We will make use of the following observation in the construction of the adjoint pair $f^* \dashv f_*$. Suppose we have a correspondence of prestacks
\[
\xymatrix{
	& C \ar[ld]_f \ar[dr]^g \\
	X && \matheur{Y}
}
\]
where $C$ and $X$ are schemes. Then we have the following pair of adjoint functors
\[
	g_! f^*: \Shv(X) \rightleftarrows \Shv(\matheur{Y}): f_* g^!
\]

\begin{prop} \label{prop:def_f_*_f^*}
	Let $f: \matheur{X} \to \matheur{Y}$ be a schematic morphism between prestacks. Then we have a pair of adjoint functors
	\[
		f^*: \Shv(\matheur{Y}) \rightleftarrows \Shv(\matheur{X}): f_*.
	\]
	Moreover, these functors are compatible with compositions of schematic morphisms between prestacks.
\end{prop}
\begin{proof}
	We will make use of Lemma~\ref{lem:cone_alpha_alphaR}. Namely, we will define the functor $f^*$ by constructing a family of compatible functors
	\[
		l_\alpha: \Shv(Y_\alpha) \to \Shv(\matheur{X})
	\]
	where $Y_\alpha$ runs over $\Sch_{/\matheur{Y}}$, and where the transition functors on the left hand side is $(-)_!$.
	
	Similarly, we will define $f_*$ by constructing a family of compatible functors
	\[
		r_\alpha: \Shv(\matheur{X}) \to \Shv(Y_\alpha)
	\]
	where $Y_\alpha$ runs over $\Sch_{/\matheur{Y}}$, and the transition functors on the left hand side is $(-)^!$.
	
	Now, for any scheme $Y_\alpha \in \Sch_{/\matheur{Y}}$, consider the following Cartesian square
	\[
	\xymatrix{
		X_\alpha \ar[d]_{f_\alpha} \ar[r]^{i'_\alpha} & \matheur{X} \ar[d]^f \\
		Y_\alpha \ar[r]^{i_\alpha} & \matheur{Y}
	}
	\]
	Since $f$ is schematic, $X_\alpha$ is a scheme, so the pair of maps $f_\alpha$ and $i'_\alpha$ forms a correspondence of the type described in the observation above.
	
	We define $l_\alpha = i'_{\alpha!}f_\alpha^*$ and $r_\alpha = f_{\alpha *}i'^!_\alpha$. By base change theorems for schemes, we see that these functors do form a family compatible with the respective transition functors. Moreover, for each $\alpha$, $l_\alpha \dashv r_\alpha$, and hence, $f^* \dashv f_*$. It is clear from the construction that $f^*$ and $f_*$ are compatible with compositions of schematic morphisms between prestacks.
\end{proof}

\begin{rmk}
	In the case where $f: X\to Y$ is a morphism of schemes, our functors $f^*$ and $f_*$ defined above coincide with the usual pullback and pushforward of sheaves on schemes.
\end{rmk}

\subsection{Equivalences of functors} When dealing with sheaves over a scheme, some of the functors that we have defined above, i.e. $f^!, f_!, f_*$, and $f^*$, may coincide depending on special properties of $f$. We have a similar situation in the world of prestacks.

\begin{prop} \label{prop:coincidence_f_*_f_!_proper}
	Let $f: \matheur{Y}_1 \to \matheur{Y}_2$ be a proper morphism between prestacks. Then, $f_* \simeq f_!$.
\end{prop}
\begin{proof}
	Since $f$ is proper, it is, in particular, also a pseudo-proper morphism. Thus, Proposition~\ref{prop:pseudo_proper_base_change} implies that $f_!$ could be computed value wise by base changing to schemes. But this is precisely how $f_*$ is constructed in Proposition~\ref{prop:def_f_*_f^*}. Hence, $f_* \simeq f_!$.
\end{proof}

\begin{rmk}
	For a morphism between prestacks $f: \matheur{Y}_1 \to \matheur{Y}_2$, by valuative criterion for properness, the following are equivalent
	\begin{enumerate}[\quad (i)]
		\item $f$ is pseudo-proper and schematic.
		\item $f$ is proper.
	\end{enumerate}
\end{rmk}

\begin{prop} \label{prop:coincidence_f^!_f^*_etale}
	Let $f: \matheur{Y}_1 \to \matheur{Y}_2$ be an \etale{} morphism between prestacks. Then $f^! \simeq f^*$.
\end{prop}
\begin{proof}
	This is direct from the corresponding fact for schemes.
\end{proof}

\subsection{Base change results} Base change theorems play an important role in the theory of sheaves on schemes. Proposition~\ref{prop:pseudo_proper_base_change} is the analog of the proper base change theorem. In this subsection, we will discuss several analogs of the proper and smooth base change theorems.

Throughout this subsection, when stating various base change results, we will keep referring to the following Cartesian diagram of prestacks
\[
\xymatrix{
	\matheur{X}' \ar[r]^g \ar[d]_f & \matheur{X} \ar[d]^f \\
	\matheur{Y}' \ar[r]^g & \matheur{Y} 
} \teq\label{eq:cartesian_sq_base_change}
\]

\subsubsection{Base change for $f_*$ and $g^!$} Since the functor $f_*$ is defined value-wise, when $f$ is a schematic morphism between prestacks, it's easy to see that we have the following base change result.

\begin{prop} \label{prop:f_*_g^!_base_change}
	Consider the Cartesian square~\eqref{eq:cartesian_sq_base_change}, where $f$ is schematic. Then, for any $\matheur{F} \in \Shv(\matheur{X})$, we have a natural equivalence
	\[
		g^! f_* \matheur{F} \simeq f_* g^! \matheur{F}.
	\]

	As a consequence, by passing to the left adjoints, for any $\matheur{G} \in \Shv(\matheur{Y}')$, we also have a natural equivalence
	\[
		f^* g_! \matheur{G} \simeq f^* g_! \matheur{G}.
	\]
\end{prop}

\begin{cor} \label{cor:id_fully_faithful}
	Let $f: \matheur{X} \to \matheur{Y}$ be a fully faithful schematic morphism between prestacks. Then, we have the following natural equivalence
	\[
		f^! f_* \simeq f^* f_! \simeq \id_{\Shv(\matheur{X})}.
	\]
	When $f$ is proper, we also have
	\[
		f^*f_* \simeq f^* f_! \simeq f^!f_! \simeq \id_{\Shv(\matheur{X})}.
	\]
\end{cor}
\begin{proof}
	This is direct from the base change result above and the fact that $\matheur{X} \times_{\matheur{Y}} \matheur{X} \simeq \matheur{X}$.
\end{proof}

\begin{cor} \label{cor:f^*_as_colimit}
	Let $f: \matheur{X} \to \matheur{Y}$ be a schematic morphism between prestacks. Then for any $\matheur{F} \in \Shv(\matheur{X})$,
	\[
		f^* \matheur{F} \simeq \colim i_{\alpha!} f_\alpha^* \matheur{F}_\alpha,
	\]
	where $\matheur{Y} \simeq \colim_\alpha Y_\alpha$ and the morphisms are named as in the following Cartesian diagram
	\[
	\xymatrix{
		X_\alpha \ar[r]^{i_\alpha} \ar[d]_{f_\alpha} & \matheur{X} \ar[d]^f \\
		Y_\alpha \ar[r]^{i_\alpha} & \matheur{Y}
	}
	\]
\end{cor}
\begin{proof}
	We have
	\[
		f^* \matheur{F} \simeq f^* \colim_\alpha i_{\alpha!} \matheur{F}_\alpha  \simeq \colim_\alpha f^* i_{\alpha!} \matheur{F}_\alpha \simeq \colim_\alpha i_{\alpha!} f_\alpha^* \matheur{F}_\alpha,
	\]
	where the first, second and third equivalences are due to~\ref{rmk:colim_ev_ins_sheaves}, continuity of $f^*$, and Proposition~\ref{prop:f_*_g^!_base_change} respectively.
\end{proof}

\subsubsection{} Note that the maps in Proposition~\ref{prop:f_*_g^!_base_change} above is not from some adjunction. Rather, it is an extra structure coming from the construction of $f_*$ and $f^*$.

\subsubsection{Base change for $f_!$ and $g^!$ for pseudo-smooth morphism $g$} We will now consider an analog of the smooth base change theorem in the world of prestacks. 

\begin{prop} \label{prop:pseudo_smooth_base_change}
	Consider the Cartesian square~\eqref{eq:cartesian_sq_base_change}, where $\matheur{Y}$ and $\matheur{Y}'$ are pseudo-Artin prestacks, and $g$ is pseudo-smooth. Then, for any $\matheur{F} \in \Shv(\matheur{X})$, we have a natural equivalence
	\[
		f_! g^! \matheur{F} \simeq g^! f_! \matheur{F}.
	\]
\end{prop}

We will build up the result from a couple of special cases.

\begin{lem} \label{lem:special_case_pseudo_smooth_base_change}
	Proposition~\ref{prop:pseudo_smooth_base_change} holds in the case where $\matheur{Y}$ and $\matheur{Y}'$ are schemes, and $g$ is a smooth morphism.
\end{lem}
\begin{proof}
	We rename the base prestacks $\matheur{Y}$ and $\matheur{Y}'$ in~\eqref{eq:cartesian_sq_base_change} to $S$ and $S'$ respectively to remind us that these are schemes. Now, for any scheme $T\in \Sch_{/\matheur{X}}$, we have the following diagram, where all squares are Cartesian
	\[
	\xymatrix{
		\matheur{X}'\times_\matheur{X} T \ar@/_1pc/[dd]_{f_T} \ar[d]^{i_T} \ar[r]^>>>>>{g_T} & T \ar[d]_{i_T} \ar@/^1pc/[dd]^{f_T} \\
		\matheur{X}' \ar[d]^f \ar[r]^g & \matheur{X} \ar[d]_f \\
		S' \ar[r]^g & S
	}
	\]
	Note that since
	\[
		\matheur{X} \simeq \colim_{T\in \Sch_{/\matheur{X}}} T,
	\]
	we have
	\[
		\matheur{X}' \simeq \colim_{T\in \Sch_{/\matheur{X}}} \matheur{X}'\times_{\matheur{X}} T \simeq \colim_{T\in \Sch_{/\matheur{X}}} S'\times_S T. \teq \label{eq:pull_back_top_colimit_over_base}
	\]

	We thus have
	\begin{align*}
		g^! f_! \matheur{F} \simeq g^! \colim_{T \in \Sch_{/\matheur{X}}} f_{T!}i_T^! \matheur{F} \simeq \colim_{T\in \Sch_{/\matheur{X}}} g^! f_{T!} i_T^! \matheur{F} \simeq \colim_{T\in \Sch_{/\matheur{X}}} f_{T!} g_T^! i_T^! \matheur{F} \simeq \colim_{T\in \Sch_{/\matheur{X}}} f_{T!} i_T^! g^! \matheur{F} \simeq f_! g^! \matheur{F}.
	\end{align*}
	Here, the first, second, third, and last equivalences are due to Lemma~\ref{lem:compute_f_!_base_scheme}, continuity of $g^!$, the usual smooth base change theorem for schemes, and Lemma~\ref{lem:compute_f_!_base_scheme} and~\eqref{eq:pull_back_top_colimit_over_base}, respectively.
\end{proof}

\begin{cor} \label{cor:pseudo_smooth_base_change_cor_pt-wise}
	Proposition~\ref{prop:pseudo_smooth_base_change} holds in the case where $\matheur{Y}$ is a pseudo-Artin prestack and $\matheur{Y}'$ is a scheme. In other word, the functor $f_!$ could be computed value-wise for pseudo-smooth morphisms from a scheme to $\matheur{Y}$.
\end{cor}
\begin{proof}
	Since $\matheur{Y}$ is a pseudo-Artin prestack, we have
	\[
		\matheur{Y} \simeq \colim_{Y_\alpha \in \Sch_{\sm/\matheur{Y}}} Y_\alpha.
	\]
	Lemma~\ref{lem:special_case_pseudo_smooth_base_change} then allows us to apply Lemma~\ref{lem:family_limit_ptwise_adjunction} to conclude.
\end{proof}

\begin{proof}[Proof of Proposition~\ref{prop:pseudo_smooth_base_change}]
	Proposition~\ref{prop:pseudo_smooth_base_change} is a now a direct consequence of Corollary~\ref{cor:pseudo_smooth_base_change_cor_pt-wise} since $f_!$ can be computed value-wise in $\Sch_{\sm/\matheur{Y}}$.
\end{proof}

We conclude the section with a corollary of the base change results that we have discussed so far. 

\begin{prop} \label{prop:devissage}
	Suppose we have a diagram of prestacks
	\[
	\xymatrix{
		\matheur{U} \ar[r]^j & \matheur{Y} & \ar[l]_i \matheur{Z}
	}
	\]
	where $\matheur{U}$ is an open sub-prestack and $\matheur{Z}$ its closed complement. Then, for any $\matheur{F} \in \Shv(\matheur{Y})$, we have the following exact triangle
	\[
		i_! i^! \matheur{F} \to \matheur{F} \to j_*j^! \matheur{F} \to \cdots
	\]
\end{prop}
\begin{proof}
	For any $X \in \Sch_{/\matheur{Y}}$, we have the following diagram obtained by pulling back
	\[
	\xymatrix{
		\matheur{U}_X \ar[r]^{j_X} \ar[d]^{f_\matheur{U}} & X \ar[d]^f & \ar[l]_{i_X} \matheur{Z}_X \ar[d]^{f_{\matheur{Z}}} \\
		\matheur{U} \ar[r]^j & \matheur{Y} & \ar[l]_i \matheur{Z}
	}
	\]
	The existence of the triangle on $\Shv(\matheur{Y})$ is equivalent to the existence of a compatible family of triangles
	\[
		f^! i_! i^! \matheur{F} \to f^! \matheur{F} \to f^! j_* j^! \matheur{F} \to \cdots
	\]
	for any $X\in \Sch_{/\matheur{Y}}$. But this is equivalent to having the following exact triangle
	\[
		i_{X!} i_X^! f^! \matheur{F} \to f^! \matheur{F} \to j_{X*} j_X^! f^! \matheur{F} \to \cdots
	\]
	for each scheme $X \in \Sch_{/\matheur{Y}}$, by the base change results, Propositions~\ref{prop:f_*_g^!_base_change} and~\ref{prop:pseudo_proper_base_change}. But then we are done since this is just the usual devissage triangle for sheaves on schemes.
\end{proof}

\subsection{Correspondences} Consider the category $\CorrPreStkSchAll$ of correspondences in $\PreStk$
\[
\xymatrix{
	\matheur{M} \ar[d]_v \ar[r]^h & \matheur{X} \\
	\matheur{Y}
} \teq\label{eq:corr_prestk}
\]
where $v$ is schematic. 

Let
\[
	F: \Corr(\Sch)_{\all;\all} \to \CorrPreStkSchAll
\]
be the natural inclusion of functor. Then we can construct a functor
\[
	\Shv: \CorrPreStkSchAll \to \DGCat_{\pres, \cont} \teq\label{eq:shv_prestk_correspondences}
\]
by the procedure of taking Right Kan Extension of~\eqref{eq:shv_sch_correspondences} along $F$. 

Restricting to $\PreStk_{\iso;\all} \simeq \PreStk^\op$, we obtain a functor
\[
	\Shv: \PreStk^\op \to \DGCat_{\pres, \cont}
\]
which agrees with our construction of sheaves on prestacks above, due to~\cite[Thm. V.2.6.1.5]{gaitsgory_study_????}. 

Since~\eqref{eq:shv_prestk_correspondences} encodes base change isomorphism for $(-)^!$ and $(-)_*$, Proposition~\ref{prop:f_*_g^!_base_change} implies that the image of~\eqref{eq:corr_prestk} under~\eqref{eq:shv_prestk_correspondences} is precisely
\[
	v_*h^!: \Shv(\matheur{X}) \to \Shv(\matheur{Y}).
\]
In other words, the functor $(-)_*$ obtained from restricting~\eqref{eq:shv_prestk_correspondences} to $\CorrPreStkSchAll \simeq \PreStk$ agrees with the functor $(-)_*$ defined in Proposition~\ref{prop:def_f_*_f^*}.

\subsection{Monoidal structure} Recall that the functor $\Shv$ on schemes has a right lax symmetric monoidal structure (see \S\ref{subsec:sheaves_on_schemes}). In this section, we upgrade this structure to sheaves on prestacks, and then study its behavior with respect to $f^*$ and $f^!$.

\subsubsection{Right lax monoidal structure of $\Shv$} The structure on $\Shv$ we are after is a functor
\[
	\boxtimes: \Shv(\matheur{Y}_1) \otimes \Shv(\matheur{Y}_2) \to \Shv(\matheur{Y}_1 \times \matheur{Y}_2)
\]
where $\matheur{Y}_1$ and $\matheur{Y}_2$ are prestacks. Given $\matheur{F}_i \in \Shv(\matheur{Y}_i)$, we define
\[
	\matheur{F}_1 \boxtimes \matheur{F}_2 \in \Shv(\matheur{Y}_1 \times \matheur{Y}_2)
\]
be such that for any scheme $S$ equipped with a morphism $(f_1, f_2): S\to \matheur{Y}_1\times \matheur{Y}_2$, 
\[
	(f_1, f_2)^! \matheur{F}_1 \boxtimes \matheur{F}_2 = f_1^! \matheur{F}_1 \otimesshriek f_2^! \matheur{F}_2.
\]
Or in the notation of \S\ref{subsubsec:def_sheaves_prestacks},
\[
	(\matheur{F}_1 \boxtimes \matheur{F}_2)_{S, (f_1, f_2)} = (\matheur{F}_{1})_{S, f_1} \otimesshriek (\matheur{F}_2)_{S, f_2}.
\]

\subsubsection{} Since we define $\boxtimes$ using $\otimesshriek$, this monoidal struture is well-behaved with respect to $!$-pullback.
\begin{lem} \label{lem:f^!_boxtimes}
	Let $f_i: \matheur{X}_i \to \matheur{Y}_i$ be morphisms between prestacks, and let $\matheur{F}_i \in \Shv(\matheur{Y}_i)$, where $i\in \{1, 2\}$. Then, we have a natural equivalence
	\[
		(f_1 \times f_2)^! \matheur{F}_1\boxtimes \matheur{F}_2 \simeq f_1^! \matheur{F}_1 \boxtimes f_2^! \matheur{F}_2.
	\]
\end{lem}

\subsubsection{} As in the case of schemes, the diagonal map $\matheur{Y} \to \matheur{Y} \times \matheur{Y}$ equips $\Shv(\matheur{Y})$ with a symmetric monoidal structure. Namely, we have
\begin{align*}
	\otimesshriek: \Shv(\matheur{Y}) \otimes \Shv(\matheur{Y}) &\to \Shv(\matheur{Y}) \\
	(\matheur{F}, \matheur{G}) &\mapsto \matheur{F}\otimesshriek \matheur{G} := \delta^! (\matheur{F} \boxtimes \matheur{G}).
\end{align*}

It is immediately from the definition that this monoidal structure is compatible with the $!$-pullback functor.

\subsubsection{} We will now list a couple of situations where a \Kunneth{} formula holds.

\begin{lem} \label{lem:Kunneth_scheme_target}
	Let $f_i: \matheur{Y}_i \to S_i$ be morphisms from prestacks to schemes, and $\matheur{F}_i\in \Shv(\matheur{Y}_i)$ where $i\in \{1, 2\}$. Then, we have a natural equivalence
	\[
		(f_1\times f_2)_! (\matheur{F}_1\boxtimes \matheur{F}_2) \simeq f_{1!} \matheur{F}_1 \boxtimes f_{2!} \matheur{F}_2.
	\]
\end{lem}
\begin{proof}
	This comes directly from the usual \Kunneth{} formula for schemes, and the fact that $f_!$ is continuous.
\end{proof}

Using~\ref{prop:pseudo_proper_base_change}, this implies the following lemma.
\begin{lem} \label{lem:Kunneth_pseudo_proper}
	Let $f_i: \matheur{X}_i \to \matheur{Y}_i$ be pseudo-proper morphisms between prestacks, and let $\matheur{F}_i \in \Shv(\matheur{X}_i)$, where $i \in \{1, 2\}$. Then, we have a natural equivalence
	\[
		(f_1\times f_2)_!(\matheur{F}_1 \boxtimes \matheur{F}_2) \simeq f_{1!} \matheur{F}_1 \boxtimes f_{2!} \matheur{F}_2.
	\]
\end{lem}
\begin{proof}
	This comes directly from the usual \Kunneth{} formula, Proposition~\ref{prop:pseudo_proper_base_change} and the continuity of $f_!$.
\end{proof}

\subsubsection{} There is another situation where a \Kunneth{} type formula holds. Let
\[
	F_i, G_i: \matheur{K} \to \Sch
\]
be functors, $\alpha_i: F_i \Rightarrow G_i$ be natural transformations, where $i\in \{1, 2\}$. Suppose that for any $m: k \to l$, the following diagram commutes
\[
\xymatrix{
	\Shv(F_i(l)) \ar[d]^{\alpha_{i, l!}} \ar[r]^{F_i(m)^!} & \Shv(F_i(k)) \ar[d]^{\alpha_{i, k!}} \\
	\Shv(G_i(l)) \ar[r]^{G_i(m)^!} & \Shv(G_i(k))
}
\]
Denote $\matheur{X}_i = \colim_{k\in \matheur{K}} F_i(k)$, $\matheur{Y}_i = \colim_{k\in \matheur{K}} G_i(k)$, and $f_i: \matheur{X}_i \to \matheur{Y}_i$ induced by $\alpha_i$.
\begin{lem} \label{lem:Kunneth_equivariant}
	In the situation described above, for any $\matheur{F}_i\in \Shv(\matheur{X}_i)$, we have
	\[
		(f_1 \times f_2)_!(\matheur{F}_1 \boxtimes \matheur{F}_2) \simeq f_{1!} \matheur{F}_1 \boxtimes f_{2!} \matheur{F}_2.
	\]
\end{lem}
\begin{proof}
	Lemma~\ref{lem:family_limit_ptwise_adjunction} implies that $f_{i!}$ and $(f_1\times f_2)_!$ can be computed term-wise using the $\alpha_{k!}$'s. But we have \Kunneth{} formula for schemes, so we are done.
\end{proof}

\subsubsection{} Since $\boxtimes$ for prestacks is defined via the shriek pullback, its interaction with the (usual) pullback is not as straightforward. However, we have the following relation.
\begin{prop} \label{prop:f^*_boxtimes}
	Let $f_i: \matheur{X}_i \to \matheur{Y}_i$ be schematic morphisms between prestacks, such that $\matheur{Y}_i$'s are pseudo-schemes, where $i \in \{1, 2\}$. Furthermore, let $\matheur{F}_i \in \Shv(\matheur{Y}_i)$. Then we have the following natural equivalence
	\[
		(f_1\times f_2)^* (\matheur{F}_1 \boxtimes \matheur{F}_2) \simeq f_1^* \matheur{F}_1 \boxtimes f_2^* \matheur{F}_2.
	\]
\end{prop}
\begin{proof}
	Since the $\matheur{Y}_i$'s are pseudo-schemes, we can write
	\[
		\matheur{Y}_i \simeq \colim_{\alpha} Y_{i\alpha}.
	\]
	where the $Y_{i\alpha}$'s are schemes, and all transition maps are proper. By~\cite[Prop. 7.4.2]{gaitsgory_atiyah-bott_2015}, we know that for any $\alpha$, the natural map $Y_{i\alpha} \to \matheur{Y}_i$ is pseudo-proper.
	
	Consider the following Cartesian diagram
	\[
	\xymatrix{
		X_{1\alpha} \times X_{2\beta} \ar[d]_{f_{1\alpha}\times f_{2\beta}} \ar[r]^{i_{1\alpha} \times i_{2\beta}} & \matheur{X}_1 \times \matheur{X}_2 \ar[d]^{f_1\times f_2} \\
		Y_{1\alpha} \times Y_{2\beta} \ar[r]^{i_{1\alpha} \times i_{2\beta}} & \matheur{Y}_1 \times \matheur{Y}_2
		}
	\]
	where all object appearing on the left column are schemes and all horizontal maps are pseudo-proper. We have
	\begin{align*}
		(f_1\times f_2)^* (\matheur{F}_1\boxtimes \matheur{F}_2)
		&\simeq \colim_{\alpha, \beta} (i_{1\alpha} \times i_{2\beta})_! (f_{1\alpha} \times f_{2\beta})^* (\matheur{F}_{1\alpha} \boxtimes \matheur{F}_{2\beta}) \\
		&\simeq \colim_{\alpha, \beta} (i_{1\alpha} \times i_{2\beta})_! (f_{1\alpha}^* \matheur{F}_{1\alpha} \boxtimes f_{2\beta}^* \matheur{F}_{2\beta}) \\
		&\simeq \colim_{\alpha, \beta} i_{1\alpha!} f_{1\alpha}^* \matheur{F}_{1\alpha} \boxtimes i_{2\beta!}f_{2\beta}^* \matheur{F}_{2\beta} \\
		&\simeq \colim_\alpha i_{1\alpha!} f^*_{1\alpha} \matheur{F}_{1\alpha} \boxtimes \colim_\beta i_{2\beta!}f_{2\beta}^* \matheur{F}_{2\beta} \\
		&\simeq f_1^* \matheur{F}_1 \boxtimes f_2^* \matheur{F}_2.
	\end{align*}
	
	Here, the first and third equivalences are due to Corollary~\ref{cor:f^*_as_colimit} and Lemma~\ref{lem:Kunneth_pseudo_proper} respectively.
\end{proof}

\subsection{Functor compositions} Consider the following Cartesian square of schemes
\[
\xymatrix{
	X_T \ar[r]^g \ar[d]_f & X \ar[d]^f \\
	T \ar[r]^g & S
}
\]
For any $\matheur{F} \in \Shv(X_T)$, we have a natural map 
\[
	g_! f_* \matheur{F} \to f_* g_! \matheur{F}. \teq \label{eq:nat_transf_square_composition_*_!}
\]
Indeed, by adjunction, giving such map is the same as giving the following map
\[
	f_* \matheur{F} \to g^! f_* g_! \matheur{F} \simeq f_* g^! g_! \matheur{F}.
\]
This map is $f_*$ applied to the unit of the $g_! \dashv g^!$ adjunction.

The following result is clear for schemes, since the two pushforward functors agree when the map is proper.
\begin{lem} \label{lem:composition_*_!_square_schemes}
	When either $f$ and $g$ is proper, the natural map at~\eqref{eq:nat_transf_square_composition_*_!} is an equivalence.
\end{lem}

In this section, we will prove statements with the same flavor as the one above in the case of prestacks. 

\begin{prop} \label{prop:composition_*_!_square_prestk}
	Consider the following Cartesian diagram
	\[
	\xymatrix{
		\matheur{X}' \ar[d]_{f'} \ar[r]^{g'} & \matheur{X} \ar[d]^f \\
		\matheur{Y}' \ar[r]^g & \matheur{Y}
	}
	\]
	where $f$, and hence $f'$, is schematic. Then, for any $\matheur{F} \in \matheur{X}'$, we have a natural map
	\[
		n: g_! f'_* \matheur{F} \to f_* g'_! \matheur{F}.
	\]	
	
	Moreover, this map is an equivalence if one of the following conditions is satisfied:
	\begin{enumerate}[\quad (i)]
		\item $f$ is proper.
		\item $g$ is finitary pseudo-proper.
	\end{enumerate}
\end{prop}
\begin{proof}
	The existence of the natural map is seen via adjunction as in~\eqref{eq:nat_transf_square_composition_*_!}. Moreover, the fact that $n$ is an equivalence in the first case is easy, since by Proposition~\ref{prop:coincidence_f_*_f_!_proper}, we know that $f_* \simeq f_!$.
	
	For the second case, we first observe that to show that $n$ is an equivalence, it suffices to check after pulling back to an arbitrary scheme $S\in \Sch_{/\matheur{Y}}$. Proposition~\ref{prop:f_*_g^!_base_change} then allows us to reduce to the case where $\matheur{X}$ and $\matheur{Y}$ are schemes. Consider the following diagram where all squares are Cartesian
	\[
	\xymatrix{
		T'_\alpha \ar[d]^{f_\alpha} \ar[r]_{i'_\alpha} \ar@/^2ex/[rr]^{g'_\alpha} & \matheur{X}' \ar[d]^{f'} \ar[r]_g & X \ar[d]^f \\
		T_\alpha \ar[r]^{i_\alpha} \ar@/_2ex/[rr]_{g_\alpha} & \matheur{Y}' \ar[r]^g & Y
	}
	\]
	where $T_\alpha$ and hence, also $T'_\alpha$, is a scheme. For any $\matheur{F} \in \Shv(\matheur{X}')$, we conclude
	\begin{align*}
		f_* g_! \matheur{F}
		&\simeq f_* \colim_\alpha g'_{\alpha!} i'^!_\alpha \matheur{F}
		\simeq \colim_\alpha f_* g'_{\alpha!} i'^!_\alpha \matheur{F}
		\simeq \colim_\alpha g_{\alpha!} f_{\alpha*} i'^!_\alpha \matheur{F}
		\simeq \colim_\alpha g_{\alpha!} i_{\alpha}^! f'_* \matheur{F} \simeq g_! f'_* \matheur{F}.
	\end{align*}
	Here, the second equivalence is due to the fact that $f_*$ commutes with finite colimits. The third and fourth equivalences are due to Lemma~\ref{lem:composition_*_!_square_schemes}, and Proposition~\ref{prop:f_*_g^!_base_change} respectively.
\end{proof}

\begin{prop} \label{prop:composition_!_*}
	Let $f: \matheur{U} \to \matheur{Z}$ and $g: \matheur{Z} \to \matheur{X}$ be morphisms of prestacks such that $g$ is finitary pseudo-proper, $f$ and $h=g\circ f$ are schematic. Then we have the following natural equivalence
	\[
		g_! \circ f_* \simeq (g\circ f)_* = h_*.
	\]
\end{prop}
\begin{proof}[Proof (Sketch)]
	The proof is similar to the above, so we will give a sketch. By base changing over a scheme $X \in \Sch_{/\matheur{X}}$, Propositions~\ref{prop:f_*_g^!_base_change} and~\ref{prop:pseudo_proper_base_change} allow us to reduce to the case where $\matheur{U}$ and $\matheur{X}$ are schemes. Now, we present $\matheur{Z}$ as a finite colimit of schemes that are proper over $\matheur{X}$ and then run a similar manipulation as in Proposition~\ref{prop:composition_*_!_square_prestk} to conclude.
\end{proof}

\section{$E_X$-algebras} \label{sec:E_X_algebras}
Factorizable sheaves, or $E_X$-algebras, are the algebro-geometric avatar of $E_n$-algebras in topology. The goal of this section is to provide the construction of free $E_X$-algebras. We will start with a quick review of $E_n$-algebras in the topological setting. An interpretation of the construction of free $E_n$-algebras in such a way that motivates the actual construction of free $E_X$-algebras will also be discussed. Our discussion of the topological setting will be impressionistic in nature, and so we will not try to be very precise. We refer the curious reader to~\cite{lurie_higher_2016} and~\cite{ayala_factorization_2012} for a detailed discussion.

Even though the construction of free $E_X$-algebras is straightforward, it is somewhat technical to show that this construction actually has the correct properties. We will extend the notion of the chiral monoidal structure in~\cite{francis_chiral_2011} to $\Conf X$ and use it as an organizing tool to formulate various formal properties of our construction. These properties will then be proved by extensive use of the base change results discussed in the previous section. 

\subsection{A quick review of $E_n$-algebras} \label{subsec:quick_review_E_n}
This subsection merely provides the topological motivation for the construction of free $E_X$-algebras in algebraic geometry. The reader who's only interested in the actual construction can safely skip this part.

\subsubsection{} Let $\Disk_n$ be the $\infty$-category whose objects are disjoint union of $n$-dimensional discs and morphisms are embeddings.  $\Disk_n$ is equipped with a natural symmetric monoidal structure given by disjoint union $\sqcup$. For sake of concreteness, we will work with the stable $\infty$-category $\Vect_\Lambda$ of chain complexes in $\Lambda$-vectors spaces, where $\Lambda$ is some ring. Note that the homotopy category of this is the usual derived category of $\Lambda$-vector spaces. This category is also equipped with the $\otimes$-monoidal structure. An $E_n$-algebra object in $\Vect_\Lambda$ is a symmetric monoidal functor
\[
	\matheur{A}: \Disk_n \to \Vect_\Lambda.
\]
We denote by $E_n(\Vect_\Lambda)$ the category of $E_n$-algebras in $\Vect_\Lambda$, i.e.
\[
	E_n(\Vect_\Lambda) = \Fun^\otimes(\Disk_n, \Vect_\Lambda).
\]

\subsubsection{} Let $\oblv$ denote the forgetful functor
\[
	\oblv: E_n(\Vect_\Lambda) \to \Vect_\Lambda.
\]
Then, it is known that the forgetful functor $E_n(\Vect_\Lambda^\otimes)$ admits a left adjoint, which is computed by
\[
	\Free_{E_n}(M) = \bigoplus_{k\geq 1} C_*(\Conf_k \mathbb{R}^n, \Lambda) \otimes_{\Sigma_k} M^{\otimes k}, \teq\label{eq:formula_free_E_n}
\]
where $\Sigma_k$ is the symmetric group on $k$ letters. Moreover,
\[
	\Conf_k \mathbb{R}^n = \oversetsupscript{(\mathbb{R}^n)}{\circ}{k} /\Sigma_k
\]
is the \emph{unordered} configuration space of $k$ points on $\mathbb{R}^n$. Here, we use $\oversetsupscript{X}{\circ}{n}$ to denote the open subspace of $X^n$ obtained by removing the \emph{fat diagonal}. 

\subsubsection{} Formula~\eqref{eq:formula_free_E_n} could be reinterpreted in a more conceptual way. Let $\fSet^\iso$ denote the category whose objects are non-empty finite sets and morphisms are isomorphisms of finite sets. This category is equipped with the disjoint union symmetric monoidal structure. It is easy to see that
\[
	T: \Vect_\Lambda \to \Fun^\otimes(\fSet^\iso, \Vect_\Lambda),
\]
defined by
\[
	M \mapsto (I \mapsto M^{\otimes I}),
\]
is an equivalence of categories. Indeed, its inverse is given by evaluating at the singleton set.

\subsubsection{} Let
\[
	i: \fSet^\iso \to \Disk_n
\]
be the obvious inclusion of categories. Then we have the following pair of adjoint functors
\[
	\LKE_i: \Fun(\fSet^\iso, \Vect_\Lambda) \rightleftarrows \Fun(\Disk_n, \Vect_\Lambda): -\circ i
\]
where $\LKE_i$ is the left Kan extension functor (see \S\ref{subsec:LKE}). 

It could be check that this induces the a pair of adjoint functors
\[
	\LKE_i: \Fun^\otimes(\fSet^\iso, \Vect_\Lambda) \rightleftarrows \Fun^\otimes(\Disk_n, \Vect_\Lambda): \res_i.
\]
Using the fact that for any topological space $X$, viewed as an $\infty$-groupoid, hence a category, we have the following
\[
	\colim_{x\in X} \Lambda \simeq C_*(X, \Lambda),
\]
we see that for any object $M\in \Vect_\Lambda$ we have
\[
	\LKE_i(T(M)) \simeq \bigoplus_{k\geq 1} C_*(\Conf_k \mathbb{R}^n, \Lambda) \otimes_{\Sigma_k} M^{\otimes k} \simeq \Free_{E_n}(M).
\]

\subsubsection{Algebraic geometry} In what follows, we will mimic the construction motivated above in algebraic geometry. Namely, for a scheme $X$, the category $\Shv(X)$ plays the role of $\Vect_\Lambda$. We will define two categories, $\Fact(\Conf X)$ and $\Fact(\Ran X)$, which play the roles of
\[
	\Fun^\otimes(\fSet^\iso, \Vect_\Lambda)
\]
and $E_n(\Vect_\Lambda)$ respectively. We call objects in $\Fact(\Ran X)$ $E_X$-algebras or factorizable sheaves on $X$.

The $\Free_{E_X}$ functor is then the left adjoint to a natural forgetful functor
\[
	\Fact(\Ran X) \to \Shv(X),
\]
and is also constructed in two stages
\[
	\Shv(X) \simeq \Fact(\Conf X) \to \Fact(\Ran X).
\]

In \S\ref{subsec:monoidal_structures} and \S\ref{subsec:factorizable_sheaves}, we will do the necessary technical preparatory work. The actually construction will be carried out in \S\ref{subsec:free_EX}, where the first and second stages are presented in \S\ref{subsubsec:1st_stage_T_Free_E_X} and \S\ref{subsubsec:second_stage_g_!_Free_E_X} respectively. The reader who is not interested in the technical details can skip ahead to \S\ref{subsec:free_EX} and refer to \S\ref{subsec:monoidal_structures} and \S\ref{subsec:factorizable_sheaves} as necessary. In some sense, this technical set up is a tool to organize the combinatorial complexity mentioned in the introduction. 

\subsection{Chiral monoidal structures on $\Conf X$ and $\Ran X$} \label{subsec:monoidal_structures}

\subsubsection{$\Conf X$, $\Sym X$ and $\Ran X$} \label{subsubsec:various_spaces} Let $\fSet$ denote the category of nonempty finite sets, and let $\fSet^\surj$ and $\fSet^\iso$ be subcategories of $\fSet$ with the same objects, but morphisms are restricted to surjective and invertible ones respectively. For any scheme $X$, we define
\begin{align*}
	\Conf X &= \bigsqcup_{n\geq 1} \Conf_n X = \bigsqcup_{n\geq 1} \oversetsupscript{X}{\circ}{n}/\Sigma_n \simeq \colim_{I \in (\fSet^\iso)^\op} \oversetsupscript{X}{\circ}{I}, \\
	\Sym X &= \bigsqcup_{n \geq 1} \Sym^n X = \bigsqcup_{n\geq 1} X^n/\Sigma_n \simeq \colim_{I\in (\fSet^\iso)^\op} X^I, \\
	\Ran X &= \colim_{I \in (\fSet^\surj)^\op} X^I.
\end{align*}

From the definition, it's easy to see that $\Conf X$ and $\Sym X$ are pseudo-Artin prestack, and moreover, the natural inclusion
\[
	\Conf X \to \Sym X
\]
is pseudo-smooth (in fact, it's an open embedding). Moreover, $\Ran X$ is a pseudo-scheme.

For each non-empty finite set $I$, we denote
\[
	\ins_I: X^I \to \Ran X
\]
the natural map.

\subsubsection{} These definitions admit the following concrete interpretations in terms of $S$-points, where $S\in \Sch$. Indeed
\[
	(\Conf X)(S) = \{I\in \fSet, f: I \to X(S) \mid \text{graph of $f$ is disjoint}\}/\sim,
\]
where two elements $(I, f)$ and $(I', f')$ are identified if there's an isomorphism of sets $I \simeq I'$ that induces $f'$ from $f$. Note that, this is equivalent to saying that $(\Conf X)(S)$ consists of all finite subsets of $X(S)$ whose graphs are disjoint.

Similarly, $(\Sym X)(S)$ is a groupoid, where the objects are
\[
	\{I \in \fSet, f: I \to X(S)\}
\]
and morphisms $(I, f) \to (I', f')$ if there is an isomorphism $I \to I'$ that induces $f'$ from $f$. Note that the groupoid aspect of $\Conf X$ is trivial since all actions are free.

Finally, $(\Ran X)(S)$ is just the set of non-empty finite subsets of $X(S)$ (see~\cite[Prop. 4.1.3]{gaitsgory_atiyah-bott_2015}).

\subsubsection{} Following~\cite{beilinson_chiral_2004}, the chiral symmetric monoidal structure on $\Shv(\Ran X)$ has been defined in~\cite{francis_chiral_2011} and then, in a more convenient language in~\cite{raskin_chiral_2015}. We will follow the approach in~\cite{raskin_chiral_2015} and extend it slightly to $\Shv(\Conf X)$.\footnote{One can do the same for $\Sym X$. However, we don't need it in this paper.}

\subsubsection{} \label{subsubsec:general_principle_monoidal}
First, recall the following observation, which is used in the construction of the chiral monoidal structure on $\Shv(\Ran X)$ in~\cite{raskin_chiral_2015}. A (non-unital) symmetric monoidal structure on a category $\matheur{C} \in \DGCat_{\pres, \cont}$ is the same as a commutative monoid structure on $\matheur{C}$. Thus, if $\matheur{K}$ is a symmetric monoidal category, equipped with a right-lax monoidal functor
\[
	\Phi: \matheur{K} \to \DGCat_{\pres, \cont},
\]
then, $\Phi(M)$ is a symmetric monoidal category for any commutative monoid object $M$ in $\matheur{K}$. Moreover, if $M\to N$ is a morphism between commutative monoids in $\matheur{K}$ then the induced functor $\Phi(M) \to \Phi(N)$ is symmetric monoidal.

\subsubsection{}
In~\cite{raskin_chiral_2015}, $\Ran X$ is equipped with the structure of a commutative monoid in $\CorrPreStkSchAll$. Applying the right-lax monoidal functor
\[
	\Shv: \CorrPreStkSchAll \to \DGCat_{\pres, \cont}, \tag{see~\eqref{eq:shv_prestk_correspondences}}
\]
to $\Ran X$, we obtain a symmetric monoidal structure on $\Shv(\Ran X)$, which is the chiral monoidal structure.

We will now recall this construction and show how it could be adapted to the case of $\Shv(\Conf X)$. But first, we need to recall some basic facts about $\Ran X$ and $\Conf X$.

\subsubsection{} \label{subsubsec:union_pseudo_proper}
For any positive integer $n$, let
\[
	\union_n: (\Ran X)^n \to \Ran X
\]
be the morphism which is just the union on $S$-points, for any test scheme $S$. We will also use $\union$ instead of $\union_n$ when there's no confusion possible.

For any non-empty finite set $I$, we have the following Cartesian diagram
\[
\xymatrix{
	\colim_{I \surjects K = \cup_{i=1}^n K_i } X^K \ar[d] \ar[rr] && (\Ran X)^n \ar[d]^{\union} \\
	X^I \ar[rr]^{\ins_I} && \Ran X
}
\]
Here, the vertical map is induced by
\[
	I \surjects K,
\]
the horizontal map by
\[
\sqcup_{i=1}^n K_i \surjects \cup_{i=1}^n K_i = K,
\]
which results in
\[
	X^K \to \prod_{i=1}^n X^{K_i} \to (\Ran X)^n.
\]
Moreover, the diagram where we take the colimit over has objects as written, and morphisms are diagrams of the following shape
\[
\xymatrix{
	I \ar@{=}[d] \ar@{->>}[r] & K \ar@{->>}[d] \ar@{=}[r] & \cup_{i=1}^n K_i \ar@{->>}[d] \\
	I \ar@{->>}[r] & K' \ar@{=}[r] & \cup_{i=1}^n K'_i
}
\]

As a direct consequence, we get
\begin{lem}
The morphism
\[
	\union_n: (\Ran X)^n \to \Ran X
\]
is finitary pseudo-proper.
\end{lem}

\subsubsection{} For any positive integer $n$, let $(\Ran X)^n_{\disj}$ denote the open sub-prestack of $(\Ran X)^n$ consisting of all points whose graphs induced by the $n$ factors are disjoint. Let
\[
	j_n: (\Ran X)^n_\disj \to (\Ran X)^n
\]
denote the open embedding. We will also use $j$ instead of $j_n$ when there's no risk of confusion. 

Finally, for non-empty finite sets $K_1, \dots, K_n$, we will use
\[
	\left(\prod_{i=1}^n X^{K_i}\right)_{\disj}
\]
to denote the open subscheme of $\prod_{i=1}^n X^{K_i}$ such that the graphs induced by the $n$ different factors are disjoint from each other.

\begin{lem} \label{lem:union_disj_etale}
For any non-empty finite set $I$, we have the following Cartesian diagram
\[
\xymatrix{
	\displaystyle\bigsqcup_{I = \sqcup_{i=1}^n K_i} \left(\prod_{i=1}^n X^{K_i}\right)_\disj \ar[r] \ar[d] & (\Ran X)^n_\disj \ar[d]^{\union\circ j} \\
	X^I \ar[r]^{\ins_I} & \Ran X
}
\]
In particular, $\union\circ j$ is an \etale{} morphism (in fact, it's a disjoint union of open morphisms).
\end{lem}
\begin{proof}
For brevity's sake, we will only write down the proof for the case where $n=2$. Namely, we will show that the following diagram is Cartesian
\[
\xymatrix{
	\displaystyle\bigsqcup_{I = K_1 \sqcup K_2} (X^{K_1} \times X^{K_2})_\disj \ar[r] \ar[d] & (\Ran X \times \Ran X)_\disj \ar[d]^{\union\circ j} \\
	X^I \ar[r]^{\ins_I} & \Ran X
}
\]
The other cases are the parallel.

Due to \S\ref{subsubsec:union_pseudo_proper}, it suffices to show that
\[
	\left(\colim_{I \surjects K \simeq K_1 \cup K_2} X^K\right) \times_{\Ran X \times \Ran X} (\Ran X \times \Ran X)_\disj \simeq \bigsqcup_{I\simeq K_1\sqcup K_2} (X^{K_1} \times X^{K_2})_\disj.
\]
But, the left hand side is equivalent to
\[
	\colim_{I \surjects K \simeq K_1 \cup K_2}\left(X^K \times_{\Ran X \times \Ran X} (\Ran X \times \Ran X)_\disj\right).
\]
Consider the following Cartesian diagram
\[
\xymatrix{
	X^K \times_{X^{K_1} \times X^{K_2}} (X^{K_1} \times X^{K_2})_\disj \ar[r] \ar[d] & (X^{K_1} \times X^{K_2})_\disj \ar[d] \ar[r] & (\Ran X\times \Ran X)_\disj \ar[d]  \\
	X^K \ar[r] & X^{K_1} \times X^{K_2} \ar[r]^<<<<<<<<{\ins_{K_1}\times \ins_{K_2}} & \Ran X\times \Ran X
}
\]
Observe that when $K_1\cap K_2 \neq \emptyset$, we have
\[
X^K \times_{X^{K_1} \times X^{K_2}} (X^{K_1} \times X^{K_2})_\disj = \emptyset.
\]
Thus, we only need to consider the colimit over the subcategory spanned by $I \surjects K \simeq K_1 \sqcup K_2$. But now, cofinality allows us to restrict further to the subcategory consisting of $I\simeq K \simeq K_1 \sqcup K_2$, and we are done.
\end{proof}

\subsubsection{} Using Lemma~\ref{lem:union_disj_etale} above, we can equip $\Ran X$ with the structure of a commutative monoid in the category of $\CorrPreStkSchAll$, where the multiplication morphisms (see \S\ref{subsubsec:base_change_correspondences_schemes} for the notation)
\[
	(\Ran X)^n \sqto \Ran X
\]
are given by
\[
\xymatrix{
	(\Ran X)^n_{\disj} \ar[d]_{\union \circ j} \ar[r]^j & (\Ran X)^n \\
	\Ran X
}
\]

This induces on $\Shv(\Ran X)$ the structure of a symmetric monoidal category, which we will call the chiral monoidal structure. 

We will write $\Shv(\Ran X)^{\otimesch}$ when we want to emphasize the monoidal structure.

\subsubsection{} Unwinding the definition above, we see that if $\matheur{F}, \matheur{G} \in \Shv(\Ran X)$, then
\[
	\matheur{F} \otimesch\matheur{G} = (\union\circ j)_* j^!(\matheur{F}\boxtimes \matheur{G}).
\]

\subsubsection{} The story for $\Conf X$ is analogous. The commutative monoid structure on $\Conf X$ in $\CorrPreStkSchAll$ is given by the following multiplication map
\[
\xymatrix{
	(\Conf X)^n_\disj \ar[r]^j \ar[d]_\union & (\Conf X)^n \\
	\Conf X
}
\]

This induces on $\Shv(\Conf X)$ the structure of a symmetric monoidal category, which we will call the chiral monoidal structure. 

We will write $\Shv(\Conf X)^{\otimesch}$ when we want to emphasize the monoidal structure.

\subsubsection{} Unwinding the definition above, we see that if $\matheur{F}, \matheur{G} \in \Shv(\Conf X)$, then
\[
	\matheur{F} \otimesch \matheur{G} = \union_* j^! (\matheur{F} \boxtimes \matheur{G}).
\]

\subsubsection{}
We have been using the same notations (i.e. $j$ and $\union$) for both $\Ran X$ and $\Conf X$. In cases where confusion might occur, we will add a ``double prime'' to morphisms regarding $\Conf X$ (for example, $j''$ and $\union''$).

For the reader's convenience, we put all the maps in the diagram below.
\[
\xymatrix{
	(\Conf X \times \Conf X)_\disj \ar@{=}[d] \ar[dr]^{j''} \ar[rr]^{(c\times c)_\disj} && (\Sym X \times \Sym X)_{\disj} \ar[d]^{j'} \ar[r]^{(\lbar{g} \times \lbar{g})_\disj} & (\Ran X \times \Ran X)_\disj \ar[d]^j \\
	(\Conf X \times \Conf X)_\disj \ar[d]^{\union''} \ar[r]^{j''} & \Conf X \times \Conf X \ar[r]^{c\times c} & \Sym X \times \Sym X \ar[r]^{\lbar{g}\times \lbar{g}} \ar[d]^{\union'} & \Ran X \times \Ran X \ar[d]^\union \\
	\Conf X \ar[rr]^c && \Sym X \ar[r]^{\lbar{g}} & \Ran X
} \teq \label{eq:big_diagram}
\]

Note that all squares/rectangles (including the trapezoid at the top left), except the one at the bottom right and the rectangle on the right, are pull-back diagrams.

\begin{rmk} \label{rmk:chiral_otimes_lower_*}
The chiral monoidal structures could be realized slightly differently using the following natural equivalences
\begin{align*}
	\union_! \circ j_* &\simeq (\union \circ j)_*, \\
	\union''_! &\simeq \union''_*,
\end{align*}
where all the maps are the ones appearing in~\eqref{eq:big_diagram}. Moreover, we also have the following equivalences
\begin{align*}
	(\union \circ j)^* &\simeq (\union \circ j)^!, \\
	\union''^* &\simeq \union''^!.
\end{align*}

Indeed, for the pushforward functors, the statement for $\Ran X$ is from Proposition~\ref{prop:composition_!_*} and Lemma~\ref{lem:union_disj_etale} above. The statement for $\Conf X$ is due to the fact that the preimage $\union''^{-1}(\Conf_I X)$ is just a disjoint union of $\Conf_I X$. For the pullback functors, we use Proposition~\ref{prop:coincidence_f^!_f^*_etale} to conclude.
\end{rmk}

\subsection{Commutative co-algebras}
Let $\ComCoAlgch(\Conf X)$ and $\ComCoAlgch(\Ran X)$ denote the categories of commutative coalgebras in $\Shv(\Conf X)^{\otimesch}$ and $\Shv(\Ran X)^{\otimesch}$ respectively (with respect to the chiral monoidal structures).

We have a pair of adjoint functors
\[
	g_!: \Shv(\Conf X) \rightleftarrows \Shv(\Ran X): g^!. \teq\label{eq:adjoint_shv_g_!_g^!}
\]
induced by the natural inclusion
\[
	g: \Conf X \to \Ran X
\]

The main goal of this subsection is the following theorem, which upgrades the adjoint pair $g_! \dashv g^!$ to a pair of adjoint functors between the correspoding categories of commutative co-algebras.

\begin{thm} \label{thm:adjunction_coalgebras_g^!_g_!}
The adjoint pair $g_! \dashv g^!$ of~\eqref{eq:adjoint_shv_g_!_g^!} upgrades to a pair of adjoint functors
\[
	g_!: \ComCoAlgch(\Conf X) \rightleftarrows \ComCoAlgch(\Ran X): g^!.
\]
\end{thm}

This theorem is a direct consequence of the following
\begin{prop} \label{prop:g^!_g_!_monoidal}
	The functor $g^!$ is symmetric monoidal (with respect to the chiral monoidal structures on both sides), and hence, $g_!$, being a left adjoint functor, is left-lax monoidal.
\end{prop}

\subsubsection{} We will make use of \S\ref{subsubsec:general_principle_monoidal} to show that $g^!$ is symmetric monoidal. Indeed, let
\[
	g_{-}: \Ran X \sqto \Conf X \teq\label{eq:g_-_Ran_Conf_corr}
\]
be a morphism in $\CorrPreStkSchAll$ given by the following correspondence
\[
\xymatrix{
	\Conf X \ar@{=}[d] \ar[r]^g & \Ran X \\
	\Conf X
}
\]

For each positive integer $n$, $g_-$ induces a natural map in $\CorrPreStkSchAll$
\[
	g_-^n: (\Ran X)^n \sqto (\Conf X)^n.
\]

We have the following
\begin{lem}
For each positive integer $n$, we have the following commutative diagram in $\CorrPreStkSchAll$
\[
\xymatrix{
	(\Ran X)^n \ar@{~>}[d] \ar@{~>}[r] & \Ran X \ar@{~>}[d] \\
	(\Conf X)^n \ar@{~>}[r] & \Conf X
}
\]
where the horizontal maps are the multiplication maps, and the vertical ones are given by $g_-$ and $g_-^n$.
\end{lem}
\begin{proof}
This is an immediate consequence of the following Cartesian square\footnote{Note that compositions inside the category of correspondences is given by Cartesian products.}
\[
\xymatrix{
	(\Conf X)^n_\disj \ar[d]_{\union} \ar[r]^{g^n_\disj} & (\Ran X)^n_\disj \ar[d]_{\union\circ j} \\
	\Conf X \ar[r]^g & \Ran X
}
\]
\end{proof}

\begin{cor}
The morphism $g_-$ of~\eqref{eq:g_-_Ran_Conf_corr} upgrades to a moprhism between commutative monoid objects in $\CorrPreStkSchAll$.
\end{cor}

\subsubsection{} As a result, \S\ref{subsubsec:general_principle_monoidal} implies that $g^!$ is symmetric monoidal, and hence $g_!$, being a left adjoint, is left-lax monoidal. We thus obtain Proposition~\ref{prop:g^!_g_!_monoidal} and hence, Theorem~\ref{thm:adjunction_coalgebras_g^!_g_!}.

\begin{rmk}
One can define the chiral monoidal structure on $\Shv(\Sym X)$ in a similar way as above. Consider
\[
\xymatrix{
	\Conf X \ar@/_1.5ex/[rr]_{g = \lbar{g} \circ c} \ar[r]^c & \Sym X \ar[r]^{\lbar{g}} & \Ran X
}
\]
where $c$ is the natural open embedding, and $\lbar{g}$ is the obvious map. One can show that $\lbar{g}_!$ is symmetric monoidal. Moreover, using Proposition~\ref{prop:pseudo_smooth_base_change}, one can show that $c_!$ is right-lax monoidal. We do not need this fact in the current paper.
\end{rmk}

\subsection{Factorizable sheaves} \label{subsec:factorizable_sheaves}
We will now come to the definition of $E_X$-algebras, or factorizable sheaves. Such a definition has been given in~\cite{francis_chiral_2011}. We will give a more geometric definition, which is equivalent to the one given there, but which fits better into our framework so far.

\subsubsection{} Let $\matheur{F}$ be an object in $\ComCoAlg^\ch(\Ran X)$ or $\ComCoAlg^\ch(\Conf X)$. Then by definition, $\matheur{F}$ is equipped with morphisms of the form
\[
	\matheur{F} \to \matheur{F}^{\otimesch n},
\]
with various compatibility conditions. Unwinding the definitions, this is the same as
\[
	\matheur{F} \to (\union\circ j)_* j^! (\matheur{F}^{\boxtimes n}) \qquad\text{or} \qquad \matheur{F} \to \union''_* j''^! (\matheur{F}^{\boxtimes n}),
\]
respectively. By adjunction and Remark~\ref{rmk:chiral_otimes_lower_*}, this is the same as
\[
	(\union\circ j)^! \matheur{F} \to j^!(\matheur{F}^{\boxtimes n}) \qquad\text{or}\qquad \union''^! \matheur{F} \to j''^!(\matheur{F}^{\boxtimes n}), \teq \label{eq:fact_morphisms_shriek}
\]
respectively, which is equivalent to
\[
	(\union\circ j)^* \matheur{F} \to j^*(\matheur{F}^{\boxtimes n}) \qquad\text{or}\qquad \union''^* \matheur{F} \to j''^*(\matheur{F}^{\boxtimes n}), \teq \label{eq:fact_morphisms_star}
\]
respectively.

\begin{defn}
$\Fact(\Ran X)$ and $\Fact(\Conf X)$ are full subcategories of 
\[
	\ComCoAlg^\ch(\Ran X) \qquad\text{and}\qquad \ComCoAlg^\ch(\Conf X)
\]
respectively, consisting of objects where the natural maps~\eqref{eq:fact_morphisms_shriek} (or~\eqref{eq:fact_morphisms_star}) are equivalences for all $n$. We call objects in $\Fact(\Ran X)$ factorizable sheaves over $X$, or $E_X$-algebras.
\end{defn}

\begin{rmk}
	In the case of $\Fact(\Ran X)$, our definition coincides with the one given in~\cite{francis_chiral_2011}, where it is called $\Fact (X)$.
\end{rmk}

\begin{prop}
	The adjoint pair $g_! \dashv g^!$ in Theorem~\ref{thm:adjunction_coalgebras_g^!_g_!} preserves factorizability, and hence, induces an adjoint pair
	\[
		g_!: \Fact(\Conf X) \rightleftarrows \Fact(\Ran X): g^!
	\]
	between the two respective full subcategories.
\end{prop}
\begin{proof}
	The fact that $g^!$ preserves factorizability is easy to see since upper-$!$ behaves nicely with tensor products, by Lemma~\ref{lem:f^!_boxtimes}.
	
	For the case of $g_!$, let $\matheur{F} \in \Fact(\Conf X)$, we will show that $g_! \matheur{F} \in \Fact(\Ran X)$. For simplicity's sake, we will only give a proof when $n=2$. The general case is completely analogous. We have,
	\begin{align*}
		j^*(g_! \matheur{F} \boxtimes g_! \matheur{F})
		&\simeq j^* (g\times g)_! (\matheur{F} \boxtimes \matheur{F}) \\
		&\simeq (g\times g)_{\disj!} j''^* (\matheur{F} \boxtimes \matheur{F}) \\
		&\simeq (g\times g)_{\disj!} \union''^* \matheur{F} \\
		&\simeq (\union\circ j)^* g_! \matheur{F}.
	\end{align*}
	Here, the first equivalence is due to the fact that \Kunneth{} formula works for both $\lbar{g}$ and $c$, by Lemmas~\ref{lem:Kunneth_pseudo_proper} and~\ref{lem:Kunneth_equivariant} respectively. The other equivalences are due to Proposition~\ref{prop:f_*_g^!_base_change}, and the fact that $\matheur{F}$ is in $\Fact(\Conf X)$.
\end{proof}

\subsection{Free $E_X$-algebras} \label{subsec:free_EX}
We are now ready to construct the free $E_X$-algebra functor.

\subsubsection{} \label{subsubsec:1st_stage_T_Free_E_X} Let $\matheur{F} \in \Shv(X)$. Then for any $I\in \fSet$, we have $\matheur{F}^{\boxtimes I}|_{\oversetsupscript{X}{\circ}{I}} \in \Shv(\oversetsupscript{X}{\circ}{I})$. The symmetric group $\Sigma_I$ on $I$ acts on everything in sight and so by definition, we get an object in $\Shv(\Conf_I X)$. These sheaves together give us an object $T\matheur{F} \in \Shv(\Conf X)$, and so we have a functor
\[
	T: \Shv(X) \to \Shv(\Conf X).
\]

Observe that $T\matheur{F}$ has a natural structure of a factorizable commutative co-algebra object in $\Shv(\Conf X)^{\otimesch}$. Indeed, we have the following natural equivalences by construction
\[
	\union''^!(T\matheur{F}) \simeq j''^!((T\matheur{F})^{\boxtimes n}).
\]
Thus $T$ upgrades to a functor
\[
	T: \Shv(X) \to \Fact(\Conf X).
\]

Let $\delta: X \to \Conf X$ denote the obvious inclusion. The following is immediate.
\begin{lem}
	$T$ and $\delta^!$  are mutually inverse functors
	\[
		T: \Shv(X) \rightleftarrows \Fact(\Conf X): \delta^!.
	\]
\end{lem}

\subsubsection{} \label{subsubsec:second_stage_g_!_Free_E_X}
The functor $\Free_{E_X}$ is defined as the composition $\Free_{E_X} = g_! \circ T$:
\[
\xymatrix{
	\Shv(X) \ar@/_2ex/[rr]_{\Free_{E_X}} \ar@{=}[r]^<<<<<T & \Fact(\Conf X) \ar[r]^{g_!} & \Fact(\Ran X)
}
\]

By abuse of notation, we also denote
\[
	\delta: X \to \Ran X
\]
the obvious map. The following theorem, which is a direct consequence of the discussion above, concludes our construction.
\begin{thm} \label{thm:Free_E_X}
	We have an adjoint pair
	\[
		\Free_{E_X}: \Shv(X) \rightleftarrows \Fact(\Ran X): \delta^!,
	\]
	where $\Free_{E_X} \simeq g_! \circ T$.
\end{thm}

\begin{rmk}
	From now on, since there will be no risk of confusion, we will use $\Fact(X)$ to denote $\Fact(\Ran X)$ in conforming to the notation used in~\cite{francis_chiral_2011}.
\end{rmk}

\section{Alternative construction via Lie algebras} \label{sec:alternative_construction_Lie}
In this section, we will present an alternative construction of the $\Free_{E_X}$ functor which links to the world of Lie algebras. The duality between commutative coalgebras and Lie algebras, which goes by the name Koszul duality, was first developed by Quillen in~\cite{quillen_rational_1969}. It was further developed in the operadic setting by Ginzburg and Kapranov in~\cite{ginzburg_koszul_1994}. In the chiral setting, the theory chiral Koszul duality, developed by Francis and Gaitsgory in~\cite{francis_chiral_2011}, provides us with the necessary connection to Lie algebras.

We will start the section with a quick summary of this theory, and refer the reader to~\cite{francis_chiral_2011} for the proofs. After that, we will present the new construction of $\Free_{E_X}$, and then conclude by showing that the two constructions are tied together by a common grading that is natural on both sides.

\subsection{Chiral Koszul duality}
\subsubsection{} One can define another monoidal structure on $\Shv(\Ran X)$, called the $\star$-monoidal structure. This is given by
\begin{align*}
	\Shv(\Ran X) \otimes \Shv(\Ran X) &\to \Shv(\Ran X) \\
	(\matheur{F}, \matheur{G}) &\mapsto \matheur{F}\otimesstar \matheur{G} := \union_!(\matheur{F} \boxtimes \matheur{G}).
\end{align*}
We will write $\Shv(\Ran X)^{\otimesstar}$ to denote the symmetric monoidal category of sheaves on the Ran prestack equipped with the $\star$-monoidal structure.

\subsubsection{} Let
\begin{align*}
	\Liech(\Ran X) &= \Lie(\Shv(\Ran X)^{\otimesch}), \\
	\Liestar(\Ran X) &= \Lie(\Shv(\Ran X)^{\otimesstar})
\end{align*}
denote the categories of Lie algebras on $\Ran X$ with respect to the $\ch$- and $\star$-monoidal structures. Let $\Liech(X)$ and $\Liestar(X)$ be the full subcategories of $\Liech(\Ran X)$ and $\Liestar(\Ran X)$ respectively consisting of Lie algebras, whose underlying objects are supported on the diagonal. In other words, we require that the underlying objects lie in the essential image of
\[
	\delta_!: \Shv(X) \to \Shv(\Ran X).
\]

\subsubsection{} The main theorem in~\cite{francis_chiral_2011} is as follows.
\begin{thm}[Francis, Gaitsgory]
	We have mutally inverse functors
	\[
		C^\ch: \Lie^\ch(\Ran X) \rightleftarrows \ComCoAlg^\ch(\Ran X) : \Prim^\ch[-1]
	\]
	which induce an equivalence of the respective full subcategories
	\[
	\xymatrix{
		\Lie^\ch(\Ran X) \ar@<\arrdis>[rr]^<<<<<<<<<<{C^\ch} && \ar@<\arrdis>[ll]^>>>>>>>>>>{\Prim^\ch[-1]} \ComCoAlg^\ch(\Ran X) \\
		\Lie^\ch(X) \ar@{^(->}[u] \ar@<\arrdis>[rr]^{C^\ch} && \ar@<\arrdis>[ll]^{\Prim^\ch[-1]} \Fact(X) \ar@{^(->}[u]
	}
	\]
\end{thm}

The functor $C^\ch$ is called the homological Chevalley complex, and it has a simple presentation which we will make use of. Namely, for a Lie algebra $\mathfrak{g} \in \Lie^\ch(X)$, $C^\ch(\mathfrak{g})$ can be exhibited as a chain complex
\[
\xymatrix{
	\cdots \ar[r] & \Sym^{\ch, n} (\mathfrak{g}[1]) \ar[r]^<<<<<D & \Sym^{\ch, n-1} (\mathfrak{g}[1]) \ar[r]^>>>>>D & \cdots 
} \teq\label{eq:Chev_complex}
\]
where $D$ is defined in terms of the Lie bracket. For our purpose, we don't even need to know what $D$ is.

Note that we can apply the same construction for the $\otimesstar$-monoidal structure and get a functor
\[
	C^\star: \Liestar(\Ran X) \to \ComCoAlgstar(\Ran X),
\]
which has the same shape as the chain complex above, except that $\Sym^\ch$ is replaced by $\Sym^\star$.

\subsubsection{} From the construction of the monoidal structure, there exists a natural morphism $\otimesstar \to \otimesch$, which gives two forgetful functors
\begin{align*}
	\oblv^{\ch\to\star}_{\Lie}: \Liech(\Ran X) &\to \Liestar(\Ran X) \\
	\intertext{and}
	\oblv^{\star\to\ch}_{\ComCoAlg}: \ComCoAlgstar(\Ran X) &\to \ComCoAlgch(\Ran X).
\end{align*}

One can show that $\oblv^{\ch\to\star}_{\Lie}$ admits a left adjoint, the chiral envelope, denoted by
\[
	U^\ch: \Liestar(\Ran X) \to \Liech(\Ran X).
\]
Moreover, by~\cite[Thm. 6.4.2, Cor. 6.4.3]{francis_chiral_2011}, the adjoint pair $U^\ch \dashv \oblv^{\ch\to\star}_{\Lie}$ preserves the corresponding full subcategories of objects with support on the diagonals.
\begin{thm}[Francis, Gaitsgory]
	We have the following diagram of adjoint functors
	\[
	\xymatrix{
		\Liestar(\Ran X) \ar@<\arrdis>[r]^{U^\ch} & \ar@<\arrdis>[l]^{\oblv^{\ch\to\star}_{\Lie}} \Liech(\Ran X) \\
		\Liestar(X) \ar@<\arrdis>[r]^{U^\ch} \ar@{^(->}[u] & \ar@<\arrdis>[l]^{\oblv^{\ch\to\star}_{\Lie}} \Liech(X) \ar@{^(->}[u]
	}
	\]
\end{thm}

\subsubsection{} On the $\otimesstar$ side, we also have the Chevalley complex functor
\[
	C^\star: \Liestar(X) \to \ComCoAlgstar(\Ran X),
\]
and the interaction with $C^\ch$ and $U^\ch$ is given by~\cite[Prop. 6.1.2]{francis_chiral_2011} as follows.
\begin{prop}[Francis, Gaitsgory] \label{prop:Uch_Cch_Cstar}
	We have the following commutative diagram of categories
	\[
	\xymatrix{
		\Lie^\star(\Ran X) \ar[d]^{U^\ch} \ar[r]^<<<<<{C^\star} & \ComCoAlgstar(\Ran X) \ar[d]^{\oblv^{\star\to\ch}_{\ComCoAlg}} \\
		\Lie^\ch(\Ran X) \ar[r]^<<<<<{C^\ch} & \ComCoAlgch(\Ran X)
	}
	\]
\end{prop}

Unlike the Chevalley complex, we don't have an intimate access to $U^\ch$. Thus, this proposition allows us to get a handle on $C^\ch\circ U^\ch$.

\subsubsection{} Let $\matheur{F}$ be a sheaf on $\Ran X$. Recall that we define
\[
	C^*_c(\Ran X, \matheur{F}) = \pi_! \matheur{F},
\]
where $\pi$ is the structure map
\[
	\pi: \Ran X \to \Spec k.
\]
Thanks to Lemma~\ref{lem:Kunneth_scheme_target}, $C^*_c(\Ran X, -)$ is monoidal with respect to $\otimesstar$ on $\Shv(\Ran X)$ and $\otimes$ on $\Vect_\Lambda$ (note that $\Shv(\Spec k) \simeq \Vect_\Lambda$). Thus, in particular, when $\matheur{F} \in \Liestar(\Ran X)$, $C^*_c(\Ran X, \matheur{F})$ has a natural structure as a Lie algebra.

We have the following result (which is~\cite[Prop. 6.3.6]{francis_chiral_2011}).
\begin{prop}[Francis, Gaitsgory] \label{prop:chiral_homology_Chevalley_complex}
	We have the following commutative diagram of categories
	\[
	\xymatrix{
		\Liestar(\Ran X) \ar[d]^{C^*_c(\Ran X, -)} \ar[r]^{U^\ch} & \Liech(\Ran X) \ar[r]^<<<<<{C^\ch} & \ComCoAlgch(\Ran X) \ar[d]^{C^*_c(\Ran, -)} \\
		\Lie(\Vect_\Lambda) \ar[rr]^{\oblv_{\ComCoAlg}\circ C} && \Vect_\Lambda 
	}
	\]
\end{prop}

\subsection{$\Free_{E_X}$ via $\Free_{\Lie}$}
\subsubsection{} The category $\Shv(X)$ is a symmetric monoidal category with respect to the $\otimes$-tensor product of sheaves. Thus, we can define
\[
	\Lie(X) = \Lie(\Shv(X)^\otimes).
\]

Recall that $\delta: X \to \Ran X$ is the natural map. Since this morphism is schematic, we have a pair of adjoint functors $\delta^* \dashv \delta_*$.

\begin{lem}
	The functor $\delta^*$ is symmetric monoidal, with respect to the $\otimesstar$-monoidal structure on $\Shv(\Ran X)$ and the $\otimes$-monoidal structure on $\Shv(X)$.	
\end{lem}
\begin{proof}
	We have the following Cartesian square
	\[
	\xymatrix{
		X \ar[r]^<<<<<{\Delta} \ar@{=}[d] & X\times X \ar[r]^<<<<<{\delta\times\delta} & \Ran X \times \Ran X \ar[d]^\union \\
		X \ar[rr]^\delta && \Ran X
	}
	\]
	Let $\matheur{F}, \matheur{G} \in \Shv(\Ran X)$. Then
	\[
		\delta^*(\matheur{F} \otimesstar \matheur{G}) \simeq \delta^* \union_! (\matheur{F}\boxtimes \matheur{G}) \simeq \Delta^* (\delta\times\delta)^*(\matheur{F}\boxtimes \matheur{G}) \simeq \Delta^*(\delta^*\matheur{F}\boxtimes \delta^* \matheur{G}) \simeq \delta^*\matheur{F}\otimes \delta^*\matheur{G},
	\]
	where the second and third equivalences are due to Proposition~\ref{prop:f_*_g^!_base_change} and Proposition~\ref{prop:f^*_boxtimes}.
\end{proof}

The fact that $\delta^*$ is monoidal implies that $\delta_*$ is right-lax monoidal. Thus, we have a pair of adjoint functors
\[
	\delta^*: \Liestar(\Ran X) \rightleftarrows \Lie(X): \delta_*.
\]

\begin{prop}[Beilinson, Drinfel'd]
	The adjoint pair $\delta^* \dashv \delta_*$ induces an equivalence of categories
	\[
		\Liestar(X) \simeq \Lie(X).
	\]
\end{prop}
\begin{proof}
	Since $\delta^*\delta_* \simeq \id$ by Corollary~\ref{cor:id_fully_faithful}, we know that $\delta_*$ is fully-faithful. But it's clear that $\delta_*$ is essentially surjective. Thus, we are done.
\end{proof}

\subsubsection{} Observe that we have the following commutative diagram
\[
\xymatrix{
	\Lie(X) \ar[dr]_{C^*_c(X, -)} \ar@{=}[r]^{\delta_! \simeq \delta_*}_{\delta^*} & \Liestar(X) \ar[d]^{C^*_c(\Ran X, -)} \\
	& \Lie(\Vect_\Lambda)
}
\]
since $C^*_c(\Ran X, -)$ is symmetric monoidal with respect to the $\star$-monoidal structures. Thus, 
\[
	C^*_c(X, -): \Lie(X) \to \Vect_\Lambda
\]
factors through
\[
	C^*_c(X, -): \Lie(X) \to \Lie(\Vect_\Lambda).
\]
And hence, we have the following corollary to Proposition~\ref{prop:chiral_homology_Chevalley_complex}.

\begin{cor} \label{cor:chiral_homology_Chevalley_complex_Lie(X)}
	We have the following commutative diagram of categories
	\[
	\xymatrix{
		\Lie(X) \ar[d]^{C^*_c(X, -)} \ar[rr]^<<<<<<<<<<<{\oblv_{\ComCoAlg}^{\star\to\ch} \circ C^\star \circ \delta_*} && \Fact(\Ran X) \ar[d]^{C^*_c(\Ran X, -)} \\
		\Lie(\Vect_\Lambda) \ar[rr]^{C} && \Vect_\Lambda
	}
	\]
\end{cor}

\subsubsection{} Consider the following diagram
\[
\xymatrix{
	\Shv(X) \ar@<-\arrdis>[d]_{\Free_{\Lie}\circ [-1]} \ar@{=}[r]^<<<<<T_<<<<<{\delta^!} & \Fact(\Conf X) \ar@<\arrdis>[r]^{g_!} & \ar@<\arrdis>[l]^{g^!} \Fact(\Ran X) \ar@<\arrdis>[d]^{\Prim[-1]} \\
	\Lie(X) \ar@<-\arrdis>[u]_{\oblv_{\Lie}[1]} \ar@{=}[r]_{\delta_*}^{\delta^*} & \Liestar(X) \ar@<\arrdis>[r]^{U^\ch} & \ar@<\arrdis>[l]^{\oblv^{\ch \to\star}_{\Lie}} \Liech(X) \ar@<\arrdis>[u]^{C^\ch}
}
\]

Since the compositions of the right adjoint functors in this diagram in both directions are $!$-restriction to the diagonal
\[
	\delta: X \to \Ran X,
\]
the right adjoint functors form a commutative diagram. The same is thus true for the left adjoint functors. In particular, we have the following statement, which provides an alternative way to construct $\Free_{E_X}$.

\begin{prop} \label{prop:2_ways_E_X}
	For any $\matheur{F} \in \Shv(X)$, we have a natural equivalence
	\[
		\Free_{E_X} \matheur{F} \simeq C^\ch (\Free_{\Lie} (\matheur{F}[-1])) \simeq \oblv_{\ComCoAlg}^{\star\to\ch}C^\star(\Free_{\Lie} (\matheur{F}[-1])).
	\]
\end{prop}
\begin{proof}
	The first equivalence is due to the discussion above, whereas the second one is due to Proposition~\ref{prop:Uch_Cch_Cstar}.
\end{proof}

We will finish this subsection with the following example where the Lie structure on $C^*_c(X, -)$ is easy to compute.
\begin{expl} \label{expl:example_Lie_constant_sheaves} Let $f$ denote the structure map
	\[
		f: X \to \Spec k.
	\]
	The fact that
	\[
		C^*_c(X, -) = f_! : \Shv(X) \to \Vect_\lambda
	\]
	is right lax monoidal with respect to the $\otimes$-monoidal structures on both $\Shv(X)$ and $\Vect_\Lambda$ provides an alternative way to see how this functor automatically upgrades to a functor on the level of Lie algebras.
	
	Since $f^*$ is symmetric monoidal, $f^*$ also upgrades to a functor
	\[
		f^*: \Lie(\Vect_\Lambda) \to \Lie(X).
	\]
	Moreover, one can check easily that $f^*$ commutes with $\Free_{\Lie}$, i.e. the following diagram commutes
	\[
	\xymatrix{
		\Shv(X) \ar[r]^<<<<<{\Free_{\Lie}} & \Lie(X) \\
		\Vect_\Lambda \ar[u]_{f^*} \ar[r]^<<<<<<{\Free_{\Lie}} & \Lie(\Vect_\Lambda) \ar[u]_{f^*}
	}
	\]
	Indeed, this is because the right adjoints commute.
	
	Let $V\in \Vect_\Lambda$, and let $V_X$ denote the constant sheaf on $X$ with value in $V$. Then by projection formula
	\[
		f_! \Free_{\Lie} V_X \simeq f_! \Free_{\Lie} f^* V \simeq f_! f^* \Free_{\Lie} V \simeq f_!(\Lambda \otimes f^* \Free_{\Lie} V) \simeq  f_!\Lambda\otimes \Free_{\Lie} V.
	\]
	Now, by cup product, $f_!\Lambda$ is a commutative dg-algebra, and the Lie structure on $f_! \Free_{\Lie} V_X$ is thus just the following
	\[
		[a\otimes v, b\otimes w] = (-1)^{|v||b|} ab [v, w],
	\]
	since all the Lie/algebra structures in sight are induced by the fact that $f_!$ is right-lax monoidal, and the equivalence above is compatible with this right-lax structure. We refer the reader to~\cite[Sect. IV.2.1.2]{gaitsgory_study_????} for a discussion of how to tensor a commutative algebra with a Lie algebra.
	
	Note that it's crucial that we are working in characteristic $0$ here. Indeed, a priori, $f_!\Lambda$ is an $E_\infty$-algebra. However, in characteristic 0, by~\cite{kriz_operads_????}, one can functorially replace it by a commutative dg-algebra.
\end{expl}

\subsection{Gradings}
Each of the two constructions of free $E_X$-algebras given above carries a natural grading. The main result of this subsection states that these two gradings are the same.

We will start with a general remark on graded objects (see also~\cite[Sect. IV.2.1.3]{gaitsgory_study_????}). This allows us to construct two natural gradings on free $E_X$-algebras, attached to the two constructions given above.

\subsubsection{} Let $\matheur{C}$ be any cocomplete symmetric monoidal stable $\infty$-category such that colimits distribute over the tensor product. Then we denote $\matheur{C}_{\gr}$ the category of graded objects in $\matheur{C}$. More formally, we regard
\[
	\mathbb{N} = \{1, 2, 3, \cdots\}
\]
as a discrete category, i.e. the only morphisms are the identities, and define
\[
	\matheur{C}_\gr = \Fun(\mathbb{N}, \matheur{C}).
\]
Thus, an object in $\matheur{C}_\gr$ could be written as $(c_i)_{i\in \mathbb{N}}$, where $c_i\in \matheur{C}$.

Note that $\matheur{C}_\gr$ is equipped with a natural symmetric monoidal structure
\[
	(c_i)_{i\in \mathbb{N}} \otimes (d_i)_{i\in \mathbb{N}} \simeq \left(\bigoplus_{n=i+j} c_i\otimes c_j\right)_n.
\]
Moreover, the forgetful functor
\begin{align*}
	\oblv_\gr: \matheur{C}_\gr &\to \matheur{C} \\
(c_i)_{i\in \mathbb{N}} &\mapsto \bigoplus_i c_i
\end{align*}
is conservative, and also compatible with the monoidal structures on both sides.

\subsubsection{} \label{subsubsec:upgrade_to_the_graded_level} We can now carry out the two constructions of free $E_X$-algebras at the graded level, i.e. with $\Shv_\gr$ instead of $\Shv$. Indeed, we have the following commutative diagram
\[
\xymatrix{
	\Shv(X)_\gr \ar@<-\arrdis>[d]_{\Free_{\Lie}\circ [-1]} \ar@{=}[r]^<<<<<T_<<<<<{\delta^!} & \Fact(\Conf X)_\gr \ar@<\arrdis>[r]^{g_!} & \ar@<\arrdis>[l]^{g^!} \Fact(\Ran X)_\gr \ar@<\arrdis>[d]^{\Prim[-1]} \\
	\Lie(X)_\gr \ar@<-\arrdis>[u]_{\oblv_{\Lie}[1]} \ar@{=}[r]_{\delta_*}^{\delta^*} & \Liestar(X)_\gr \ar@<\arrdis>[r]^{U^\ch} \ar[ur]^<<<<<<<{\oblv^{\star \to \ch}\circ C^\star} & \ar@<\arrdis>[l]^{\oblv^{\ch \to\star}_{\Lie}} \Liech(X)_\gr \ar@<\arrdis>[u]^{C^\ch}
} \teq\label{eq:2_constructions_graded_level}
\]

Now, the compositions of the right adjoint functors commute for the same reason as before. The same is thus true for the compositions of the left adjoint functors. It's also easy to check that $\oblv_\gr$ commutes with the compositions of the left adjoint functors, so that the construction at the graded level agrees with the non-graded version. In other words, the equivalence of the two constructions of free $E_X$-algebras upgrades to an equivalence at the graded level.


\subsubsection{Cardinality grading} Let $\matheur{F}\in \Shv(X)$. Then we can view $\matheur{F} \in \Shv(X)_\gr$ by putting $\matheur{F}$ at degre 1 (and zero everywhere else). Following the top row of~\eqref{eq:2_constructions_graded_level}, we obtain a natural grading on
\[
	\Free_{E_X} \matheur{F} \simeq g_!(T(\matheur{F}))
\]
which we will call the \emph{cardinality grading}.

More concretely, as an element in $\Shv(\Conf X)$, we can write
\[
	T\matheur{F} \simeq \bigoplus_{n \geq 1} (T \matheur{F})_n \teq\label{eq:grading_TF}
\]
where
\[
	\Supp (T\matheur{F})_n \subset \Conf_n X.
\]
is a sheaf on $\Conf X$, whose support is $\Conf_n X$ (see \S\ref{subsubsec:1st_stage_T_Free_E_X} for the definition of $T\matheur{F}$).

Thus, as a sheaf on $\Ran X$, we also have the following decomposition into a direct sum
\[
	\Free_{E_X} \matheur{F} \simeq \bigoplus_{n\geq 1} (\Free_{E_X} \matheur{F})_n, \teq\label{eq:decomposition_cardinality}
\]
where
\[
	(\Free_{E_X} \matheur{F})_n \simeq g_! (T\matheur{F})_n.
\]

\subsubsection{Lie grading} Again, let $\matheur{F} \in \Shv(X)$ and view it as an object of $\Shv(X)_\gr$ as above. Following the bottom row, we obtain another natural grading on
\[
	\Free_{E_X} \matheur{F} \simeq C^\ch (U^\ch (\delta_* \Free_{\Lie}(\matheur{F}[-1]))) \simeq \oblv^{\star \to \ch}(C^\star (\delta_* (\Free_{\Lie}(\matheur{F}[-1]))))
\]
which we will call the \emph{Lie grading}.

More concretely, $\Free_{\Lie} (\matheur{F}[-1])$ has a natural structure as a graded Lie algebra
\[
	\Free_{\Lie} (\matheur{F}[-1]) \simeq \bigoplus_{w\geq 1} (\Free_{\Lie} (\matheur{F}[-1]))_w, \teq\label{eq:grading_Free_Lie}
\]
where
\[
	(\Free_{\Lie} (\matheur{F}[-1]))_1 \simeq \matheur{F}[-1].
\]

The functor $C^\star$ then induces a natural grading on the resulting object and we have
\[
	C^\star(\Free_{\Lie}(\matheur{F}[-1])) \simeq \bigoplus_{w \geq 1} C^\star(\Free_{\Lie}(\matheur{F}[-1]))_w.
\]

\begin{rmk} This decomposition has the following concrete realization, though we will not need this in the sequel. Applying the same reasoning as above to the functor $\Sym^n$ for any $n\geq 1$, we see that
\[
	\Sym^n \mathfrak{g} = \bigoplus_{w \geq 1} (\Sym^n \mathfrak{g})_{w},
\]
where the case $w=n$ has a particularly nice description
\[
	(\Sym^n \mathfrak{g})_n = \Sym^n (\mathfrak{g}_1).
\]
Informally, $(\Sym^n \mathfrak{g})_w$ is spanned by monomials of degree $n$ in $\mathfrak{g}$ such that the sum of the weights is $w$.

This gives us a concrete presentation of $C(\mathfrak{g})_w$ as a chain complex as follows (see also~\eqref{eq:Chev_complex})
\[
\xymatrix{
	0 \ar[r] & \Sym^{w} (\mathfrak{g}_1[1]) \ar[r] & (\Sym^{w-1} (\mathfrak{g}[1]))_w \ar[r] & \cdots \ar[r] & g_w[1] \ar[r] & 0.
}
\]
In the case where $\mathfrak{g} = \Free_{\Lie} (\matheur{F}[-1])$, $C^\star(\Free_{\Lie}(\matheur{F}[-1]))_w$ has the form
\[
\xymatrix{
	0 \ar[r] & \Sym^{\star, w}(\matheur{F}) \ar[r] & (\Sym^{\star, w-1}(\mathfrak{g}[1]))_w \ar[r] & \cdots \ar[r] & \mathfrak{g}_w[1] \ar[r] & 0.
} \teq \label{eq:Chevalley_weight_w}
\]
\end{rmk} 

\subsubsection{} The following result follows immediately from the discussion in \S\ref{subsubsec:upgrade_to_the_graded_level}.

\begin{thm} \label{thm:compatibility_of_gradings}
	The equivalence given by Proposition~\ref{prop:2_ways_E_X} exchanges the cardinality grading and the Lie grading.
\end{thm}

\section{Consequences} \label{sec:consequences}
We will now list several consequences of the constructions carried out above. These results are phrased completely in classical terms.

\subsection{Homology of $\Conf X$}
Our construction provides a mean to compute homology of configuration spaces.

\begin{prop} \label{prop:coh_of_conf_spaces}
	Let $X$ be a scheme, and $\matheur{F} \in \Shv(X)$. Then, there exists a functorial quasi-isomorphism of chain complexes
	\[
		C^*_c(\Conf X, T\matheur{F}) = \bigoplus_{n\geq 1} C^*_c(\Conf_n X, (T\matheur{F})_n) \simeq C^{\Lie}_*(C^*_c(X, \Free_{\Lie}(\matheur{F}[-1]))),
	\]
	which exchanges the cardinality grading (of configuration) on the left hand side with the Lie grading on the right hand side. Here $C_*^{\Lie}$ denotes the homological Chevalley complex functor.
\end{prop}
\begin{proof}
	Due to the functoriality of pushforward with compact support, the construction of $\Free_{E_X}$ implies that
	\[
		C^*_c(\Conf X, T\matheur{F}) \simeq C^*_c(\Ran X, \Free_{E_X} \matheur{F}).
	\]
	The desired equivalence is now a direct consequence of Proposition~\ref{prop:2_ways_E_X}, Theorem~\ref{thm:compatibility_of_gradings} and Corollary~\ref{cor:chiral_homology_Chevalley_complex_Lie(X)}.
\end{proof}

When $\matheur{F} \simeq \omega_X$ is the dualizing sheaf on $X$, where $X$ is smooth, we recover~\cite[Thm. 1.1]{knudsen_betti_2014} using Example~\ref{expl:example_Lie_constant_sheaves} and the formality result~\cite[Sect. 7.1]{knudsen_betti_2014}. Indeed, this is the content of Corollary~\ref{cor:coh_of_conf_spaces_formality_incorporated}. But first, we need some preparation.

\subsubsection{} Let $\mathfrak{g}$ be a dg-Lie algebra. Then
\[
	H_*^{\Lie}(\mathfrak{g}) = H_*(C_*^{\Lie}(\mathfrak{g}))
\]
is computed as the homology of a double complex, and hence could be computed using a spectral sequence (see~\cite[Ex. 2.2]{manetti_formality_2015}). This endows $H_*^{\Lie}(\mathfrak{g})$ with a filtration functorially depending only on $\mathfrak{g}$. 

When $\mathfrak{g}$ is formal, i.e. when there exists a (not necessarily canonical) equivalence 
\[
	\mathfrak{g} \simeq H^*(\mathfrak{g})
\]
of dg-Lie algebras, we have an equivalence (again, not necessarily canonical)
\[
	H^{\Lie}_*(\mathfrak{g}) \simeq H^{\Lie}_*(H^*(\mathfrak{g})). \teq \label{eq:non-canonical_equivalence_Chev_hom}
\]
But by loc. cit., page 2 of the spectral sequence mentioned above is just $H^{\Lie}_*(H^*(\mathfrak{g}))$. The equivalence~\eqref{eq:non-canonical_equivalence_Chev_hom} thus implies that page 2 is equal to page $\infty$, and we obtain the following result.

\begin{lem}
	Let $\mathfrak{g}$ be a formal dg-Lie algebra. Then, there's a natural filtration on $H_*^{\Lie}(\mathfrak{g})$ that gives rise to a natural equivalence
	\[
		\gr_* H_*^{\Lie} (\mathfrak{g}) \simeq H_*^{\Lie}(H^*(\mathfrak{g})),
	\]
	where $H^{\Lie}_*(\mathfrak{g}) \simeq H_*(C^{\Lie}_*(\mathfrak{g}))$.
\end{lem}

\begin{cor} \label{cor:coh_of_conf_spaces_formality_incorporated}
	Let $X$ be a smooth scheme of dimension $d$. Then there exists a natural filtration on $H^*_c(\Conf X, \omega_{\Conf X})$ which gives rise to a natural equivalence
	\[
		\gr_* H^*_c(\Conf X, \omega_{\Conf X}) \simeq H_*^{\Lie} (H^*_c(X, \Qlbar) \otimes \Free_{\Lie} (\Qlbar[2d-1](d)))
	\]
	In general, for any scheme $X$, we have
	\[
		\gr_* H^*_c(\Conf X, \Qlbar) \simeq H^{\Lie}_*(H^*_c(X, \Qlbar) \otimes \Free_{\Lie} (\Qlbar[-1])).
	\]
\end{cor}

\begin{rmk}
	Note that $\Free_{\Lie} (\Qlbar[-1])$ is a very simple object. Indeed, we have
	\[
		\Free_{\Lie} (\Qlbar[-1]) \simeq \Qlbar[-1]\oplus \Qlbar[-2], \teq\label{eq:free_Lie_simple}
	\]
	where the only non-zero bracket is the isomorphism
	\[
		(\Qlbar[-1])^{\otimes 2} \simeq \Qlbar[-2].
	\]
\end{rmk}

\begin{rmk}
	There is a small difference between this statement and the statement found in~\cite{knudsen_betti_2014}. Namely, the equivalence we give is natural, which is desirable since we want to talk about compatibility with Galois actions.
\end{rmk}

Results in~\cite{knudsen_betti_2014} could now be done internally in the world of algebraic geometry. For example, the proof found in~\cite[Sect. 5.3]{knudsen_betti_2014}, which is an argument about the homology of a Lie algebra (namely, the one on the right hand side of Proposition~\ref{prop:coh_of_conf_spaces}), could now be copied without any modification to yield a proof in our context. As a consequence, we get

\begin{cor} \label{cor:homological_stability}
	For a connected smooth scheme $X$ of dimension $n$, cap product with the unit in $H^0(X, \Lambda)$ induces a map
	\[
		H^*_c(\Conf_{k+1} X, \omega_{\Conf_{k+1} X}) \to H^*_c(\Conf_k X, \omega_{\Conf_k X})
	\]
	that is 
	\begin{enumerate}[\quad --]
		\item an isomorphism, for $*>-k$, and a surjection for $*=-k$, when $X$ is an algebraic curve; and
		\item an isomorphism for $*\geq -k$, and a surjection for $*=-k-1$ in all other cases.
	\end{enumerate}
\end{cor}

\subsection{Galois action} One merit of our approach is the fact that Galois actions are already part of the output. More precisely, let $X_0$ be a scheme over $k_0$, $\matheur{F}_0 \in \Shv(X_0)$. Moreover, let $X$ and $\matheur{F}$ be the base changes of $X_0$ and $\matheur{F}_0$ to $k = \lbar{k}_0$. Then, $C^*_c(\Conf X, T\matheur{F})$ automatically carries the action of the Galois group $\Gal (k/k_0)$.

Since the equivalence given in Proposition~\ref{prop:coh_of_conf_spaces} is natural, it's compatible with the action of the Galois group on both sides. This implies, for instance, the equivalence in Corollary~\ref{cor:coh_of_conf_spaces_formality_incorporated} and the stabilizing maps in Corollary~\ref{cor:homological_stability} are compatible with the Galois actions. Note that in the case where $k_0 = \Fq$, Corollary~\ref{cor:coh_of_conf_spaces_formality_incorporated} gives a direct way to read off the Frobenius weights of the cohomology of configuration spaces.

\subsection{A simple example} As a simple example, we will now compute the cohomology of configuration spaces of $\mathbb{A}^n$. This same computation has been done in~\cite[Sect. 6.1]{knudsen_betti_2014}; all we are saying is that the same computation also yields information about Frobenius weights. We have
\[
	H^*_c(\Conf \mathbb{A}^n, \Qlbar) \simeq H^{\Lie}_*(H^*_c(\mathbb{A}^n, \Qlbar) \otimes \mathfrak{g})
\]
where $\mathfrak{g}$ is a Lie algebra given in~\eqref{eq:free_Lie_simple}. Note that a priori, this equivalence is only true when the LHS is replaced by the associated graded with respect to some canonical filtration. The equivalence above is still valid, because, as we shall see, all the cohomology groups have dimension 1.

Now
\[
	H^*_c(\mathbb{A}^n, \Qlbar) \simeq \Qlbar[-2n](-n)
\]
with no cup product, and thus, the Lie algebra structure on
\[
	H^*_c(\mathbb{A}^n, \Qlbar) \otimes \mathfrak{g}
\]
is trivial.

This implies that
\[
	H^{\Lie}_*(H^*_c(\mathbb{A}^n, \Qlbar) \otimes \mathfrak{g}) \simeq \Sym (H^*_c(\mathbb{A}^n, \Qlbar) \otimes \mathfrak{g}[1]) \simeq \left(\Qlbar[u] \otimes \bigwedge v\right)_+
\]
with no differential, where  $u$ and $v$ are generators living in
\[
	\Qlbar[-2n](-n) \qquad \text{and} \qquad \Qlbar[-2n-1](-n)
\]
respectively, and the plus sign denotes the augmentation ideal, i.e. we remove the base $\Qlbar$ (note that our $\Sym$ is non-unital). Moreover, in terms of the cardinality gradings, $u$ has degree $1$ and $v$ has degree $2$. Thus, when $d \geq 2$, we have
\[
	H^*_c(\Conf_d \mathbb{A}^n, \Qlbar) \simeq u^{d} \Qlbar \oplus u^{d-2} v \Qlbar.
\]
Namely,
\[
	H^i_c(\Conf_d \mathbb{A}^n, \Qlbar) \simeq
	\begin{cases}
		\Qlbar(-nd) & i = 2nd, \\
		\Qlbar(-n(d-1)) & i = 2n(d-1)+1.
	\end{cases}
\]

\begin{rmk}
	In \cite{knudsen_betti_2014}, there are many worked out examples, which would be transported using the same procedure as above.
\end{rmk}

\appendix
\section{Categorical remarks}
In this appendix, we collect, along with references, several categorical results that are used throughout the paper. 

\subsection{Left Kan extension} \label{subsec:LKE} Let $\matheur{C}, \matheur{C}'$ and $\matheur{D}$ be $\infty$-categories, where $\matheur{D}$ is cocomplete. Let
\[
	\alpha: \matheur{C} \to \matheur{C}'
\]
be a functor. Then, we have the following pair of adjoint functors
\[
	\LKE_\alpha: \Fun(\matheur{C}, \matheur{D}) \to \Fun(\matheur{C}', \matheur{D}): \res_\alpha,
\]
where $\res_\alpha$ is restricting along $\alpha$, and $\LKE_\alpha$ is the left Kan extension along $\alpha$. The latter functor is computed as follows: for any functor
\[
	F: \matheur{C}\to \matheur{D},
\]
and for any $c'\in \matheur{C}'$,
\[
	(\LKE_\alpha F)(c') = \colim_{c\in \matheur{C}_{/c'}} F(c).
\]
Here, $\matheur{C}_{/c'}$ fits into the following Cartesian square of categories
\[
\xymatrix{
	\matheur{C}_{/c'} \ar[d] \ar[r] & \matheur{C} \ar[d]^\alpha \\
	\matheur{C}'_{/c'} \ar[r] & \matheur{C}'
}
\]

\subsection{Limits vs. colimits} \label{subsec:limit_vs_colimit} Let
\[
	F: \matheur{K} \to \DGCatprescont
\]
be a diagram. By the adjoint functor theorem (see~\cite[Corollary 5.5.2.9]{lurie_higher_2012}), there exists a functor
\[
	F^R: \matheur{K}^\op \to \DGCatpres,
\]
where we replace all the arrows by their right adjoints. Denote $\matheur{C}_k = F(k) = F^R(k)$. We have
\[
	\lim_{k\in \matheur{K}^\op} \matheur{C}_k \simeq \colim_{k\in \matheur{K}} \matheur{C}_k \teq \label{eq:limit_vs_colimit}.
\]
It is important to note that the colimit is taken inside $\DGCatprescont$ and the limit is taken in $\DGCatpres$. Note also that the inclusion
\[
	\DGCatprescont\to \DGCatpres
\]
preserves limits but not necessarily colimits (see~\cite[Prop. 5.5.3.13]{lurie_higher_2012}).

For any $k\in \matheur{K}$, if we let
\[
	\ins_k: \matheur{C}_k \to \colim_{k\in \matheur{K}} \matheur{C}_k \qquad \text{and} \qquad \ev_k: \lim_{k\in \matheur{K}^\op} \matheur{C}_k \to \matheur{C}_k
\]
be the obvious functors, then, in the case where $F^R$ also factors through
\[
	F^R: \matheur{K}^\op \to \DGCatprescont,
\]
the equivalence at~\eqref{eq:limit_vs_colimit} can be realized concretely as
\[
	x \simeq \colim_{k \in \matheur{K}} \ins_k \ev_k x. \teq\label{eq:explicit_equivalence_lim_vs_colim}
\]

\subsection{Adjunctions}
\subsubsection{} Let $\matheur{C} \in \DGCatprescont$, and
\[
	F: \matheur{K} \to \DGCatprescont,
\]
and
\[
	\alpha: \colim_{k\in \matheur{K}} F(k) \to \matheur{C}
\]
a functor in $\DGCatprescont$. Let $F^R: \matheur{K}^\op\to \DGCatpres$ be the diagram obtained from $F$ by replacing all functors by their right adjoints. Denote $\alpha_k: F(k) \to \matheur{C}$ the functor induced by $\alpha$, and $\alpha_k^R$ its right adjoint. Then, the $\alpha_k^R$'s determine a functor
\[
	\alpha^R: \matheur{C} \to \lim_{k\in \matheur{K}^\op} F^R(k).
\]
Again, by \S\ref{subsec:limit_vs_colimit}, the limit and colimit appearing above are the same. The following lemma is the content of \S\ref{subsec:limit_vs_colimit}.
\begin{lem} \label{lem:cone_alpha_alphaR}
	We have an adjoint pair $\alpha\dashv \alpha^R$.
\end{lem}

\subsubsection{} Let $F, G: \matheur{K} \to \Cat$ be two functors, $\alpha: F\Rightarrow G$ and $\alpha^L: G\Rightarrow F$ two natural transformations, such that for each $k\in \matheur{K}$, $\alpha_k^L \dashv \alpha_k$. By the universal property of limits, $\alpha$ and $\alpha^L$ induce functors 
\[
	\alpha^L: \lim_{k\in \matheur{K}} F(k) \rightleftarrows \lim_{k\in \matheur{K}} G(k): \alpha.
\]

We have the following result from~\cite[Lem. I.1.2.6.4]{gaitsgory_study_????}.

\begin{lem} \label{lem:family_limit_ptwise_adjunction}
	We have an adjoint pair $\alpha^L \dashv \alpha$.
\end{lem}

\bibliography{ConfX} 

\begin{bibdiv}
\begin{biblist}

\bib{ayala_factorization_2012}{article}{
      author={Ayala, D.},
      author={Francis, J.},
       title={Factorization homology of topological manifolds},
        date={2012-06-24},
      eprint={1206.5522},
         url={http://arxiv.org/abs/1206.5522},
}

\bib{ayala_poincare/koszul_2014}{article}{
      author={Ayala, D.},
      author={Francis, J.},
       title={{Poincaré/Koszul} duality},
        date={2014-09-08},
      eprint={1409.2478},
         url={http://arxiv.org/abs/1409.2478},
}

\bib{beilinson_chiral_2004}{book}{
      author={Beilinson, A.},
      author={Drinfeld, V.},
       title={Chiral algebras},
      series={American Mathematical Society Colloquium Publications},
   publisher={American Mathematical Society, Providence, {RI}},
      volume={51},
        ISBN={0-8218-3528-9},
         url={http://www.ams.org/mathscinet-getitem?mr=2058353},
      review={\MR{2058353}},
}

\bib{francis_chiral_2011}{article}{
      author={Francis, J.},
      author={Gaitsgory, D.},
       title={Chiral koszul duality},
        date={2011-03-29},
         url={http://arxiv.org/abs/1103.5803},
        note={Selecta Math. (N.S.) 18 (2012), no. 1, 27-87},
}

\bib{farb_etale_2015}{article}{
      author={Farb, B.},
      author={Wolfson, J.},
       title={Etale homological stability and arithmetic statistics},
        date={2015-12-01},
      eprint={1512.00415},
         url={http://arxiv.org/abs/1512.00415},
}

\bib{gaitsgory_atiyah-bott_2015}{article}{
      author={Gaitsgory, D.},
       title={The {Atiyah-Bott} formula for the cohomology of the moduli space
  of bundles on a curve},
        date={2015-05-09},
      eprint={1505.02331},
         url={http://arxiv.org/abs/1505.02331},
}

\bib{getzler_resolving_1999}{article}{
      author={Getzler, E.},
       title={Resolving mixed {Hodge} modules on configuration spaces},
        date={1999},
        ISSN={0012-7094, 1547-7398},
      volume={96},
      number={1},
       pages={175\ndash 203},
         url={http://projecteuclid.org/euclid.dmj/1077228945},
      review={\MR{MR1663927}},
}

\bib{ginzburg_koszul_1994}{article}{
      author={Ginzburg, V.},
      author={Kapranov, M.},
       title={Koszul duality for operads},
        date={1994},
        ISSN={0012-7094, 1547-7398},
      volume={76},
      number={1},
       pages={203\ndash 272},
         url={http://projecteuclid.org/euclid.dmj/1077286744},
      review={\MR{MR1301191}},
}

\bib{gaitsgory_weils_2014}{article}{
      author={Gaitsgory, D.},
      author={Lurie, J.},
       title={Weil's conjecture for function fields},
        date={2014-09},
         url={http://www.math.harvard.edu/~lurie/papers/tamagawa.pdf},
}

\bib{gaitsgory_study_????}{book}{
      author={Gaitsgory, D.},
      author={Rozenblyum, N.},
       title={A study in derived algebraic geometry},
         url={http://math.harvard.edu/~gaitsgde/GL/},
}

\bib{kriz_operads_????}{book}{
      author={Kriz, I.},
      author={May, J.~P.},
       title={Operads, algebras, modules, and motives},
}

\bib{kupers_$e_n$-cell_2014}{article}{
      author={Kupers, A.},
      author={Miller, J.},
       title={{$E_n$}-cell attachments and a local-to-global principle for
  homological stability},
        date={2014-05-27},
      eprint={1405.7087},
         url={http://arxiv.org/abs/1405.7087},
}

\bib{kupers_homological_2013}{article}{
      author={Kupers, A.},
      author={Miller, J.},
       title={Homological stability for topological chiral homology of
  completions},
        date={2013-11-20},
         url={http://arxiv.org/abs/1311.5203},
}

\bib{knudsen_betti_2014}{article}{
      author={Knudsen, B.},
       title={Betti numbers and stability for configuration spaces via
  factorization homology},
        date={2014-05-26},
      eprint={1405.6696},
         url={http://arxiv.org/abs/1405.6696},
}

\bib{lurie_higher_2016}{book}{
      author={Lurie, J.},
       title={Higher algebra},
}

\bib{lurie_higher_2012}{book}{
      author={Lurie, J.},
       title={Higher topos theory},
      series={Annals of Mathematics Studies},
      number={170},
}

\bib{liu_enhanced_2014}{article}{
      author={Liu, Y.},
      author={Zheng, W.},
       title={Enhanced adic formalism and perverse $t$-structures for higher
  artin stacks},
        date={2014-04-03},
         url={http://arxiv.org/abs/1404.1128},
}

\bib{liu_enhanced_2012}{article}{
      author={Liu, Y.},
      author={Zheng, W.},
       title={Enhanced six operations and base change theorem for artin
  stacks},
        date={2012-11-26},
      eprint={1211.5948},
         url={http://arxiv.org/abs/1211.5948},
}

\bib{manetti_formality_2015}{article}{
      author={Manetti, M.},
       title={On some formality criteria for {DG}-lie algebras},
        date={2015-09-15},
        ISSN={0021-8693},
      volume={438},
       pages={90\ndash 118},
  url={http://www.sciencedirect.com/science/article/pii/S0021869315002331},
}

\bib{quillen_rational_1969}{article}{
      author={Quillen, D.},
       title={Rational homotopy theory},
        date={1969-09-01},
      volume={90},
      number={2},
       pages={205\ndash 295},
         url={http://www.jstor.org/stable/1970725},
}

\bib{raskin_chiral_2015}{article}{
      author={Raskin, S.},
       title={Chiral categories},
        date={2015-01-23},
         url={https://dash.harvard.edu/handle/1/12274305},
}

\bib{totaro_configuration_1996}{article}{
      author={Totaro, B.},
       title={Configuration spaces of algebraic varieties},
        date={1996-10-01},
        ISSN={0040-9383},
      volume={35},
      number={4},
       pages={1057\ndash 1067},
  url={http://www.sciencedirect.com/science/article/pii/0040938395000585},
}

\end{biblist}
\end{bibdiv}
\end{document}